\documentclass[a4paper,11pt]{article}

\usepackage{amsthm}
\usepackage{amsmath}
\usepackage{amssymb}
\usepackage{graphicx}
\usepackage[round]{natbib}
\usepackage{hyperref}
\usepackage{bbm}
\usepackage{textcomp}

\usepackage{color}
\usepackage{framed}
\usepackage{comment}
\definecolor{shadecolor}{gray}{0.9}

\usepackage{dsfont} 

\specialcomment{extra}{\begin{shaded}}{\end{shaded}}

\theoremstyle{plain}  
\newtheorem{thm}{Theorem}[section] 
\newtheorem{lem}[thm]{Lemma} 
\newtheorem{prop}[thm]{Proposition} 
\newtheorem{cor}[thm]{Corollary} 

\theoremstyle{definition} 
\newtheorem{defn}[thm]{Definition}

\theoremstyle{remark} 
\newtheorem{rem}[thm]{Remark}

\newtheoremstyle{assumption}
{3pt}
{3pt}
{}
{}
{\bf}
{.}
{.5em}
{\thmname{#1} (\thmnote{#3}\thmnumber{#2})}

\theoremstyle{assumption}
\newtheorem{ass}{Assumption}

\newcommand{\diff}{\,\mathrm{d}}
\newcommand{\dint}{\mathrm{d}}
\newcommand{\E}{\mathbb{E}}

\newcommand{\conv}{\operatorname{conv}}
\newcommand{\ES}{\operatorname{ES}}

\newcommand{\essinf}{\operatorname{ess\,inf}}

\newcommand{\F}{\mathcal{F}}
\newcommand{\R}{\mathbb{R}}
\newcommand{\A}{\mathsf{A}}
\renewcommand{\O}{\mathsf{O}}
\newcommand{\one}{\mathds{1}}
\newcommand{\interior}{\operatorname{int}}
\newcommand{\mat}{\operatorname{mat}}
\newcommand{\rank}{\operatorname{rank}}

\newcommand{\ph}{\varphi}

\newcommand{\VaR}{\operatorname{VaR}}
\newcommand{\argmin}{\operatorname{arg\,min}}

\def\be{\begin{equation} \label}
\def\ee{\end{equation}}

\numberwithin{equation}{section} 

\begin{document}

\title{Higher order elicitability and Osband's principle}
\author{Tobias Fissler\thanks{University of Bern, Department of Mathematics and Statistics, Institute of Mathematical Statistics and Actuarial Science, Sidlerstrasse 5, 3012 Bern, Switzerland, 
e-mails: \texttt{tobias.fissler@stat.unibe.ch} and \texttt{johanna.ziegel@stat.unibe.ch}} \and Johanna F.~Ziegel\footnotemark[1]}
\maketitle

\begin{abstract}
A statistical functional, such as the mean or the median, is called elicitable if there is a scoring function or loss function such that the correct forecast of the functional is the unique minimizer of the expected score. Such scoring functions are called strictly consistent for the functional. The elicitability of a functional opens the possibility to compare competing forecasts and to rank them in terms of their realized scores. In this paper, we explore the notion of
elicitability for multi-dimensional functionals and give both necessary and sufficient conditions for strictly consistent scoring functions. We cover the case of functionals with elicitable components, but we also show that one-dimensional functionals that are not elicitable can be a component of a higher order elicitable functional. In the case of the variance this is a known result. However, an important result of this paper is that spectral risk measures with a spectral measure with finite support are jointly elicitable if one adds the `correct' quantiles. A direct consequence of applied interest is that the pair (Value at Risk, Expected Shortfall) is jointly elicitable under mild conditions that are usually fulfilled in risk management applications.
\end{abstract}



\section{Introduction}
Point forecasts for uncertain future events are issued in a variety of different contexts such as business, government, risk-management or meteorology, and they are often used as the basis for strategic decisions. In all these situations, one has a random quantity $Y$ with unknown distribution $F$. One is interested in a statistical property of $F$, that is a functional $T(F)$. Here, $Y$ can be real-valued (GDP growth for next year), vector-valued (wind-speed, income from taxes for all cantons of Switzerland), functional-valued (path of the interchange rate Euro\,-\,Swiss franc over one day), or set-valued (area of rain tomorrow, area of influenza in a country). Likewise, also the functional $T$ can have a variety of different sorts of values, amongst them the real- and vector-valued case (mean, vector of moments, covariance matrix, expectiles), the set-valued case (confidence regions) or also the functional-valued case (distribution functions). This article is concerned with the situation where $Y$ is a $d$-dimensional random vector and $T$ is a $k$-dimensional functional, thus also covering the real-valued case.

It is common to assess and compare competing point forecasts in terms of a loss function or \textit{scoring function}. This is a function $S$ such as the squared error or the absolute error which is negatively oriented in the following sense: If the forecast $x\in \R^k$ is issued and the event $y\in\R^d$ materializes, the forecaster is \textit{penalized} by the real value $S(x,y)$. In the presence of several different forecasters one can compare their performances by ranking their realized scores. Hence, forecasters have an incentive to minimize their Bayes risk or \textit{expected loss} $\E_F[S(x,Y)]$. \cite{Gneiting2011} demonstrated impressively that scoring functions should be incentive compatible in that they should encourage the forecasters to issue truthful reports; see also \citet{MurphyDaan1985,EngelbergManskiETAL2009}. In other words, the choice of the scoring function $S$ must be consistent with the choice of the functional $T$. We say a scoring function $S$ is $\F$-consistent for a functional $T$ if $T(F)\in \argmin_x \E_F[S(x,Y)]$ for all $F\in \F$ where the class $\F$ of probability distributions is the domain of $T$. If $T(F)$ is the unique minimizer of the expected score for all $F\in\F$ we say that $S$ is \textit{strictly $\F$-consistent for $T$}. Hence, a strictly $\F$-consistent scoring function for $T$ elicits $T$. Following \cite{Lambert2008} and \citet{Gneiting2011}, we call a functional $T$ with domain $\F$ \textit{elicitable} if there exists a strictly $\F$-consistent scoring function for $T$. 

The elicitability of a functional allows for regression, such as quantile regression and expectile regression \citep{Koenker2005,NeweyPowell1987} and for M-estimation \citep{Huber1964}.
Early work on elicitability is due to \cite{Osband1985,OsbandReichelstein1985}. More recent advances in the one-dimensional case, that is $k = d = 1$ are due to \citet{Gneiting2011, Lambert2012, SteinwartPasinETAL2014} with the latter showing the intimate relation between elicitability and identifiability. Under mild conditions, many important functionals are elicitable such as moments, ratios of moments, quantiles and expectiles. However, there are also relevant functionals which are not elicitable such as variance, mode, or Expected Shortfall \citep{Osband1985, Weber2006, Gneiting2011, Heinrich2013}. 

	With the so-called \textit{revelation principle} (see Proposition \ref{prop: revelation principle}) \cite{Osband1985} was one of the first to show that a functional, albeit itself not being elicitable, can be a component of an elicitable vector-valued functional. The most prominent example in this direction is that the pair (mean, variance) is elicitable despite the fact that variance itself is not. However, it is crucial for the validity of the revelation principle that there is a bijection between the pair (mean, variance) and the first two moments. 
Until now, it appeared as an open problem if there are elicitable functionals with non-elicitable components other than those which can be connected to a functional with elicitable components via a bijection. \cite{FrongilloKash2014b} conjectured that this is generally not possible. We solve this open problem and can reject their conjecture: Corollary \ref{cor:VaR-ES} shows that the pair (Value at Risk, Expected Shortfall) is elicitable, subject to mild regularity assumptions, improving a recent partial result of \cite{AcerbiSzekely2014}. To the best of our knowledge, we provide the first proof of this result in full generality. 
In fact, Corollary \ref{cor:spectral} demonstrates more generally that \textit{spectral risk measures} with a spectral measure having finite support in $(0,1]$ can be a component of an elicitable vector-valued functional. 
These results may lead to a new direction in the contemporary discussion about what risk measure is best in practice, and in particular about the importance of elicitability in risk measurement contexts \citep{EmbrechtsHofert2014,EmmerKratzETAL2013,Davis2013,AcerbiSzekely2014}. 

Complementing the question whether a functional is elicitable or not, it is interesting to determine the class of strictly consistent scoring functions for a functional, or at least to characterize necessary and sufficient conditions for the strict consistency of a scoring function. Most of the existing literature focuses on real-valued functionals meaning that $k=1$. For the case $k>1$, mainly linear functionals, that is, vectors of expectations of certain transformations, are classified where the only strictly consistent scoring functions are \textit{Bregman functions} \citep{Savage1971, OsbandReichelstein1985,DawidSebastian1999, BanerjeeGuoETAL2005, AbernethyFrongillo2012}; for a general overview of the existing literature, we refer to \cite{Gneiting2011}. To the best of our knowledge, only \cite{Osband1985}, \cite{Lambert2008} and \cite{FrongilloKash2014b} investigated more general cases of functionals, the latter also treating vectors of ratios of expectations as the first non-linear functionals. In his doctoral thesis, \cite{Osband1985} established a necessary representation for the first order derivative of a strictly consistent scoring function with respect to the report $x$ which connects it with identification functions. Following \cite{Gneiting2011} we call results in the same flavor \textit{Osband's principle}. Theorem \ref{thm: Osband's principle} in this paper complements and generalizes \citet[Theorem 2.1]{Osband1985}. Using our techniques, we retrieve the results mentioned above concerning the Bregman representation, however under somewhat stronger regularity assumptions than the one in \cite{FrongilloKash2014b}; see Corollary \ref{cor:Bregman}. On the other hand, we are able to treat a much broader class of functionals; see Proposition \ref{prop:elicitable components}, Remark \ref{rem:elicitable components} and Theorem \ref{thm:spectral}. In particular, we show that under mild richness assumptions on the class $\F$, any strictly $\F$-consistent scoring function for a vector of quantiles and\,/\,or expectiles is the sum of strictly $\F$-consistent one dimensional scoring functions for each quantile\,/\,expectile; see Corollary \ref{cor:elicitable components}.

The paper is organized as follows. In Section \ref{sec:basic}, we introduce notation and derive some basic results concerning the elicitability of $k$-dimensional functionals. Section \ref{sec:Osband} is concerned with \textit{Osband's principle}, Theorem \ref{thm: Osband's principle}, and its immediate consequences. We investigate the situation where a functional is composed of elicitable components in Section \ref{sec:elicitable components}, whereas Section \ref{sec:spectral} is dedicated to the elicitability of spectral risk measures. We end our article with a brief discussion; see Section \ref{sec:discussion}. Most proofs are deferred to Section \ref{sec:proofs}.

\section{Properties of higher order elicitability}\label{sec:basic}

\subsection{Notation and definitions}\label{subsec:notation}

Following \cite{Gneiting2011}, we introduce a decision-theoretic framework for the evaluation of point forecasts. To this end, we introduce an \textit{observation domain} $\O\subseteq \R^d$. We equip $\O$ with the Borel $\sigma$-algebra $\mathcal O$ using the induced topology of $\R^d$. We identify a Borel probability measure $P$ on $(\O,\mathcal O)$ with its cumulative distribution function (cdf) $F_P\colon \O\to[0,1]$ defined as $F_P(x):=P((-\infty,x]\cap\O)$, where $(-\infty,x] = (-\infty,x_1]\times \dots\times (-\infty,x_d]$ for $x = (x_1,\dots,x_d) \in \R^d$. Let $\F$ be a class of distribution functions on $(\O,\mathcal O)$. Furthermore, for some integer $k\ge1$, let $\A\subseteq \R^k$ be an \textit{action domain}. 
To shorten notation, we usually write $F\in\F$ for a cdf and also omit to mention the $\sigma$-algebra $\mathcal O$.

Let $T\colon\F\to\A$ be a functional. We introduce the notation $T(\F):=\{x\in\A \colon x = T(F) \; \text{for some}\; F \in \F\}$. For a set $M\subseteq \R^k$ we will write $\interior(M)$ for its interior with respect to $\R^k$, that is, $\interior(M)$ is the biggest open set $U\subseteq \R^k$ such that $U\subseteq M$. The convex hull of $M$ is defined as ,
\[
\conv(M):= \Big\{ \sum_{i=1}^n \lambda_{i}x_{i} \,\big|\, n\in\mathbb N, \ x_1, \ldots, x_n\in M, \ \lambda_1, \ldots, \lambda_n>0,\ \sum_{i=1}^n \lambda_i=1 \Big\}.
\]

We say that a function $a\colon\O\to\R$ is $\F$-integrable if it is $F$-integrable for each $F\in\F$. A function $g\colon \A \times \O \to \R$ is $\F$-integrable if $g(x,\cdot)$ is $\F$-integrable for each $x\in \A$. If $g$ is $\F$-integrable, we introduce the map 
\[
\bar g\colon \A \times \F \to \R, \quad (x,F) \mapsto  \bar g(x,F) = \int g(x,y)\diff F(y).
\]
Consequently, for fixed $F\in\F$ we can consider the function $\bar g(\cdot, F)\colon \A \to \R$, $x\mapsto \bar g(x,F)$, and for fixed $x\in \A$ we can consider the (linear) functional $\bar g(x,\cdot) \colon \F \to \R$, $F\mapsto \bar g(x,F)$.

If we fix $y\in\O$ and $g$ is sufficiently smooth in its first argument, then for $m\in\{1, \ldots, k\}$ we denote the $m$-th partial derivative of the function $g(\cdot,y)$ with $\partial_m g(\cdot,y)$. More formally, we set
\[
\partial_m g(\cdot, y) \colon \interior(\A) \to\R, \quad (x_1, \ldots, x_k)\mapsto \tfrac{\partial}{\partial x_m} g(x_1, \ldots, x_k,y).
\]
We denote by $\nabla g(\cdot, y)$ the gradient of $g(\cdot, y)$ defined as $\nabla g(\cdot,y):= \big(\partial_1 g(\cdot,y),$ $ \ldots,\partial_k g(\cdot,y)\big)^\top$; and with $\nabla^2 g(\cdot,y):= \big(\partial_l\partial_m g(\cdot,y)\big)_{l,m=1, \ldots, k}$ the Hessian of $g(\cdot, y)$.
\textit{Mutatis mutandis}, we use the same notation for $\bar g(\cdot, F)$, $F\in\F$. We call a function on $\A$ differentiable if it is differentiable in $\interior(\A)$ and use the notation as given above. The restriction of a function $f$ to some subset $M$ of its domain is denoted by $f_{|M}$.

\begin{defn}[Consistency]\label{def:consistency}
A \textit{scoring function} is an $\F$-integrable function $S\colon \A\times \O \to \R$.
It is said to be \textit{$\F$-consistent} for a functional $T\colon \F \to \A$ if
\(
\bar S(T(F),F)\le \bar S(x,F)
\)
for all $F\in \F$ and for all $x\in \A$. Furthermore, $S$ is \textit{strictly $\F$-consistent} for $T$ if it is $\F$-consistent for $T$ and if
$\bar S(T(F),F) = \bar S(x,F)$ implies that $x = T(F)$
for all $F\in \F$ and for all $x\in \A$. Wherever it is convenient we assume that $S(x,\cdot)$ is locally bounded for all $x \in \A$.
\end{defn}

\begin{defn}[$k$-elicitability]
A functional $T\colon \F\to \A\subseteq\R^k$ is called \textit{$k$-elicitable}, if there exists a strictly $\F$-consistent scoring function for $T$.
\end{defn}

\begin{defn}[Identification function]
An \emph{identification function} is an $\F$-integrable function $V\colon \A\times \O \to \R^k$. It is said to be an \textit{$\F$-identification function} for a functional $T\colon \F \to \A\subseteq \R^k$ if
\(
\bar V(T(F),F)=0
\)
for all $F\in\F$. Furthermore, $V$ is a \textit{strict $\F$-identification function} for $T$ if 
$\bar V(x,F) = 0$ holds if and only if $x = T(F)$
for all $F\in \F$ and for all $x\in \A$.  Wherever it is convenient we assume that $V(x,\cdot)$ is locally bounded for all $x \in \A$ and that $V(\cdot,y)$ is locally Lebesgue-integrable for all $y \in \O$. 
\end{defn}

\begin{defn}[$k$-identifiability]
A functional $T\colon \F\to \A\subseteq\R^k$ is said to be \textit{$k$-identifiable}, if there exists a strict $\F$-identification function for $T$.
\end{defn}

If the dimension $k$ is clear from the context, we say that a functional is elicitable (identifiable) instead of $k$-elicitable ($k$-identifiable).

\begin{rem}
Depending on the class $\F$, some statistical functionals such as quantiles can be set-valued. In such situations, one can define $T\colon\F\to 2^\A$. Then, a scoring function $S\colon \A\times \O \to \R$ is called (strictly) $\F$-consistent for $T$ if $\bar S(t,F)\le \bar S(x,F)$ for all $x\in\A$, $F\in\F$ and $t\in T(F)$ (with equality implying $x \in T(F)$). The definition of a (strict) $\F$-identification function for $T$ can be generalized \textit{mutatis mutandis}.
Many of the results of this paper can be extended to the case of set-valued functionals -- at the cost of a more involved notation and analysis. To allow for a clear presentation, we confine ourselves to functionals with values in $\R^k$ in this paper.
\end{rem}

If $V\colon \A \times \O \to \R^k$ is an $\F$-identification function for a functional $T\colon \F \to \A$ and $h\colon \A\to\R^{k\times k}$ is a matrix-valued function, then the function 
\[
hV\colon \A\times\O \to \R^k, \quad (x,y)\mapsto hV(x,y) :=h(x)V(x,y)
\]
is again an $\F$-identification function for $T$. If $V$ is a strict $\F$-identification function for $T$ and $\det (h(x))\neq 0$ for all $x\in \A$, then $hV$ is also a strict $\F$-identification function for $T$. 

\begin{rem}\label{rem: Orientation}
\cite{SteinwartPasinETAL2014} introduced the notion of an \textit{oriented} strict $\F$-iden\-ti\-fi\-cation function for the case $k=1$ (and $d=1$). They say that $V\colon \A\times \O\to\R$ is an oriented strict $\F$-identification function for the functional $T\colon \F\to\A$ if $V$ is a strict $\F$-identification function for $T$ and moreover
\be{eq:inequ}
\bar V(x,F) > 0  \quad \Longleftrightarrow \quad x > T(F)
\ee
for all $F\in \F$ and for all $x\in \A$. They show -- under some regularity assumptions such as the continuity of the functional $T$ -- that if $V$ is a strict $\F$-identification function for the functional $T$ then either $V$ or $-V$ is oriented; see \citet[Lemma 6]{SteinwartPasinETAL2014}. This notion of orientation can also be generalized to the case $k>1$. 
\end{rem}

\begin{defn}[Orientation]\label{defn:orientation}
Let $T\colon \F\to \A$ be a functional with a strict $\F$-identification function $V\colon \A\times \O\to\R^k$. Then $V$ is called an \textit{oriented} strict $\F$-identification function for $T$ if
\[
v^\top \bar V(T(F) + sv,F)> 0  \quad \Longleftrightarrow \quad s>0
\]
for all $v\in\mathbb S^{k-1} := \{x\in\R^k \colon \|x\|=1\}$, for all $F\in \F$ and for all $s\in \R$ such that $T(F) + sv\in\A$.
\end{defn}
Indeed, the one-dimensional definition of orientation at \eqref{eq:inequ} is nested in Definition \ref{defn:orientation} upon recalling that $\mathbb S^0 = \{-1,1\}$. Under some smoothness assumptions, we can give a necessary condition for the orientation of a strict $\F$-identification function $V$: Assume that the function $\A\to \R^k$, $x\mapsto \bar V(x,F)$ 
is partially differentiable. If $V$ is oriented then the matrix $\big(\partial_l \bar V_r(t,F) \big)_{r,l=1, \ldots, k}$ is positive semi-definite for all $F\in\F$ and $t = T(F)$. It appears to be an open question under which conditions there exists an oriented identification function for an identifiable functional. In the light of Lemma \ref{lem:sufficiency} (ii), Remark \ref{rem:sufficiency} and Proposition \ref{prop: integration} this would give insight whether the construction of a strictly proper scoring function is possible. 

\begin{rem}
Our notion of orientation differs from the one proposed by \cite{FrongilloKash2014b}. In contrast to their definition, our definition is \textit{per se} independent of a (possibly non-existing) strictly consistent scoring function for $T$. Moreover, with respect to Lemma \ref{lem:sufficiency} (ii) and Remark \ref{rem:sufficiency}, the orientation of the gradient of a scoring function implies its strict consistency.
\end{rem}

\subsection{Basic results}

The first lemma gives a sufficient condition for strict consistency and connects the notions of scoring functions and identification functions. 

\begin{lem}\label{lem:sufficiency}
\begin{enumerate}
\item[(i)]
A scoring function $S\colon \A\times \O \to \R$ is strictly $\F$-consistent for $T:\F \to \A\subseteq \R^k$ if and only if the function
\[
\psi\colon D \to\R, \qquad s\mapsto \bar S(t+sv,F)
\]
has a global unique minimum at $s=0$ for all $F\in \F$, $t = T(F)$ and $v\in\mathbb S^{k-1}$ where $D = \{s\in\R\colon t+sv\in\A\}$.
\item[(ii)]
Let $S\colon \A\times \O\to\R$ be a scoring function that is continuously differentiable in its first argument and let $\F'= T^{-1}(\interior(\A))\subseteq \F$. If $\nabla S\colon \interior(\A)\times \O\to \R^k$ is an oriented strict $\F'$-identification function for $T_{|\F'}$ then $S_{|\interior(\A)\times \O}$ is a strictly $\F'$-consistent scoring function for $T_{|\F'}$.
\end{enumerate}
\end{lem}

\begin{rem}\label{rem:sufficiency}
One can weaken the assumptions of Lemma \ref{lem:sufficiency} (ii) on the smoothness of $S$. Let $S\colon \A\times \O\to\R$ be a scoring function such that $\bar S(\cdot,F)$ is continuously differentiable for all $F \in \F$. If $\F$ consists of absolutely continuous distributions, this is a much weaker requirement; see Section \ref{sec:Osband} for a detailed discussion. Let $\F'= T^{-1}(\interior(\A))\subseteq \F$. If for all $F'\in\F$, $t = T(F)\in\interior(\A)$, for all $v\in\mathbb S^{k-1}$ and for all $s\in \R$ such that $t + sv\in\interior(\A)$ we have that 
\[
v^\top \nabla \bar S(t+sv,F) 
\begin{cases}
>0, &\text{if } s>0 \\
=0, &\text{if } s=0 \\
<0, &\text{if } s<0
\end{cases}
\]
then $S_{|\interior(\A)\times \O}$ is a strictly $\F'$-consistent scoring function for $T_{|\F'}$.
\end{rem}

The following result follows directly from the definition of consistency (Definition \ref{def:consistency}). However, it is crucial to understand many of the results of this paper.

\begin{lem}\label{lem:consistency}
Let $T\colon \F\to\A\subseteq\R^k$ be a functional with a strictly $\F$-consistent scoring function $S\colon \A\times \O\to\R$. Then the following two assertions hold.
\begin{enumerate}
\item[(i)]
Let $\F'\subseteq \F$ and $T_{|\F'}$ be the restriction of $T$ to $\F'$. Then $S$ is also a strictly $\F'$-consistent scoring function for $T_{|\F'}$.
\item[(ii)]
Let $\A'\subseteq \A$ such that $T(\F)\subseteq \A'$ and $S_{|\A'\times \O}$ be the restriction of $S$ to $\A'\times \O$. Then $S_{|\A'\times \O}$ is also a strictly $\F$-consistent scoring function for $T$.
\end{enumerate}
\end{lem}

The main results of this paper consist of necessary and sufficient conditions for the strict $\F$-consistency of a scoring function $S$ for some functional $T$. What are the consequences of Lemma \ref{lem:consistency} for such conditions? Assume that we start with a functional $T'\colon \F'\to \A' \subseteq\R^k$ and deduce some \textit{necessary} conditions for a scoring function $S'\colon \A'\times \O\to\R$ to be strictly $\F'$-consistent for $T'$. Then Lemma \ref{lem:consistency} (i) implies that these conditions continue to be necessary conditions for the strict $\F$-consistency of $S'$ for $T\colon \F\to\A'$ where $\F'\subseteq \F$, and $T$ is some extension of $T'$ such that $T(\F)\subseteq \A'$. On the other hand, Lemma \ref{lem:consistency} (ii) implies that the necessary conditions for the strict $\F'$-consistency of a scoring function $S'\colon \A'\times \O\to\R$ continue to be necessary conditions for the strict $\F'$-consistency of $S\colon \A\times \O\to\R$ for $T'$, where $\A'\subseteq \A$ and $S$ is some extension of $S'$.

Summarizing, given a functional $T\colon \F\to\A$, a collection of necessary conditions for the strict $\F$-consistency of scoring functions for $T$ is the more restrictive the smaller the class $\F$ and the smaller the set $\A$ is (provided that $T(\F)\subseteq \A$, of course). Hence, in the forthcoming results concerning necessary conditions, it is no loss of generality to just mention which distributions must necessarily be in the class $\F$ to guarantee the validity of the results. Furthermore, it is no loss of generality to make the assumption that $T$ is surjective, so $\A = T(\F)$.

Some of the subsequent results also provide \textit{sufficient} conditions for the strict $\F$-consistency of a scoring function $S\colon \A\times \O\to\R$ for a functional $T\colon\F\to\A$.
Those results are the stronger the bigger the class $\F$ and the bigger the set $\A$ is. For the notion of \text{elicitability} this means that the assertion that a functional $T\colon \F\to\A$ is elicitable is also the stronger the bigger the class $\F$ and the bigger the set $\A$ is. To demonstrate this reasoning, observe that if the functional $T\colon \F\to \A$ is degenerate in the sense that it is constant, so $T\equiv t$ for some $t\in\A$ (which covers the particular case that $\F$ contains only one element), then $T$ is automatically elicitable with a strictly $\F$-consistent scoring function
$S\colon \A\times \O\to\R$, defined as $S(x,y):= \|x - t\|$.

Strictly consistent scoring functions for a given functional $T$ are not unique. In particular, the following result generalizes directly from the one-dimensional case. Let $S\colon \A\times \O \to \R$ be a strictly $\F$-consistent scoring function a functional $T\colon \F \to \A$. Then, for any $\lambda>0$ and any $\F$-integrable function $a\colon \O\to\R$, the scoring function
\begin{equation}\label{eq:Sequiv}
\widetilde S(x,y) := \lambda S(x,y) +a(y)
\end{equation}
is again strictly $\F$-consistent for $T$. \citet[Theorem 2]{Gneiting2011} shows that in the one-dimensional case under the assumption $S(x,y) \ge 0$, the class of consistent scoring functions is a convex cone. Generally, the assumption of scoring functions being nonnegative is natural if $\delta_y \in \F$ for all $y \in \O$ because for an $\F$-consistent scoring function $S$, the scoring function $\widetilde S(x,y) := S(x,y) - \bar S(T(\delta_y),\delta_y) \ge 0$ and it is of the form \eqref{eq:Sequiv} if $y \mapsto \bar S(T(\delta_y),\delta_y)$ is $\F$-integrable. As we are particularly interested in classes $\F$ of absolutely continuous distributions in this manuscript, we do not require scoring functions to be nonnegative. We generalize \citet[Theorem 2]{Gneiting2011} as follows showing that the class of strictly $\F$-consistent scoring functions for $T$ is a convex cone (not including zero). The proof follows easily using Fubini's theorem and is omitted.

\begin{prop}\label{prop:cone}
Let $T\colon \F\to\A \subseteq \R^k$ be a functional. Let $(Z, \mathcal Z)$ be a measurable space with a $\sigma$-finite measure $\nu$ where $\nu\neq0$. Let $\{S_z\colon z\in Z\}$ be a family of strictly $\F$-consistent scoring functions $S_z\colon \A\times \O\to \R$ for $T$. If for all $x\in\A$ and for all $F\in\F$ the map 
$Z\times \O \to \R$, $(z,y)\mapsto S_z(x,y)$,
is $\nu \otimes F$-integrable, then the scoring function
\[
S\colon \A\times \O\to\R, \qquad (x,y)\mapsto S(x,y) = \int_Z S_z(x,y)\nu(\dint z)
\]
is strictly $\F$-consistent for $T$.
\end{prop}

Point forecasts and probabilistic forecasts are closely related. Probabilistic forecasts, issuing a whole probability distribution, can be evaluated in terms of \textit{scoring rules} \citep{Winkler1996,GneitingRaftery2007}. A scoring rule is a map $R\colon \F\times \O\to\R$ such that for each $G\in\F$, the map $\O\to\R$, $y\mapsto R(G,y)$ is $\F$-integrable. A scoring rule is (strictly) $\F$-proper if $\bar R(F,F)\le \bar R(G,F)$ for all $F,G\in\F$ (with equality implying $F = G$). As in the one-dimensional case \citep[Theorem 3]{Gneiting2011}, each $\F$-consistent scoring function $S$ for a functional $T\colon \F\to\A\subseteq \R^k$ induces an $\F$-proper scoring rule $R$ via
\[
R\colon \F\times \O \to \R, \qquad (F,y)\mapsto R(F,y) = S(T(F),y).
\]
However, if we do not impose that the functional $T$ is injective, we cannot conclude that $R$ is a \textit{strictly} $\F$-proper scoring rule even if the scoring function $S$ is strictly $\F$-consistent. 

Many important statistical functionals are transformations of other statistical functionals, for example variance and first and second moment are related in this manner. The following \textit{revelation principle}, which originates from \citet[p.\,8]{Osband1985} and is also given in \citet[Theorem 4]{Gneiting2011} states that if two functionals are related by a bijection, then one of them is elicitable if and only if the other one is elicitable. The assertion also holds upon replacing `elicitable' with `identifiable'. We omit the proof which is straightforward.

\begin{prop}[Revelation principle]\label{prop: revelation principle}
Let $g\colon \A\to\A'$ be a bijection with inverse $g^{-1}$, where $\A, \A'\subseteq \R^k$. Let $T\colon \F\to\A$ be a functional. Then the following two assertions hold.
\begin{enumerate}
\item[(i)]
The functional $T\colon \F\to \A$ is identifiable if and only if $T_g = g\circ T \colon \F\to\A'$ is identifiable. The function $V\colon \A\times \O\to\R^k$ is a strict $\F$-identification function for $T$ if and only if 
\[
V_g\colon \A'\times \O\to\R^k, \qquad (x',y)\mapsto V_g(x',y) = V(g^{-1}(x'),y)
\]
is a strict $\F$-identification function for $T_g$.
\item[(ii)]
The functional $T\colon \F\to \A$ is elicitable if and only if $T_g = g\circ T \colon \F\to\A'$ is elicitable. The function $S\colon \A\times \O\to\R$ is a strictly $\F$-consistent scoring function for $T$ if and only if 
\[
S_g\colon \A'\times \O\to\R, \qquad (x',y)\mapsto S_g(x',y) = S(g^{-1}(x'),y)
\]
is a strictly $\F$-consistent scoring function for $T_g$.
\end{enumerate}
\end{prop}

We remark that also \citep[Theorem 5]{Gneiting2011} on weighted scoring functions carries over directly to the higher order case. Furthermore, convexity of level sets continues to be a necessary condition for elicitability.  The result is classical in the literature and was first presented in \citet[Proposition 2.5]{Osband1985}; see also \citet[Theorem 6]{Gneiting2011}.

\begin{prop}[Osband]\label{prop:convex level sets}
Let $T\colon \F\to\A\subseteq\R^k$ be an elicitable functional. Then for all $F_0, F_1\in\F$ with $t:=T(F_0) = T(F_1)$ and for all $\lambda \in(0,1)$ such that $F_\lambda:= (1-\lambda)F_0 + \lambda F_1\in\F$ it holds that $t=T(F_\lambda)$.
\end{prop}

As a last result in this section, we present the intuitive observation that a vector of elicitable functionals itself is elicitable. 
\begin{lem}\label{lem: sum}
Let $k_1,\ldots, k_l\ge1$ and let $T_m\colon \F \to \A_m \subseteq \R^{k_m}$ be a $k_m$-elicitable functional, $m\in\{1, \ldots, l\}$. Then the functional $T = (T_1,\ldots, T_l)\colon$ $\F\to\A$ is $k$-elicitable where $k=k_1+\cdots +k_l$ and $\A = \A_{1}\times \cdots \times \A_{l}\subseteq \R^k$.
\end{lem}
\begin{proof}
For $m\in\{1, \ldots, l\}$ let $S_m\colon \A_m\times \O \to \R$ be a strictly $\F$-consistent scoring function for $T_m$. Let $\lambda_1,\ldots, \lambda_l>0$ be positive real numbers. Then
\begin{gather}\label{eqn: sum}
S\colon \A_{1}\times \cdots \times \A_{l} \times \O\to\R, \\
(x_1,\ldots, x_l,y)\mapsto S(x_1,\ldots, x_l,y):= \sum_{m=1}^l \lambda_m S_m(x_m,y)
\nonumber
\end{gather}
is a strictly $\F$-consistent scoring function for $T$.
\end{proof}

A particularly simple and relevant case of Lemma \ref{lem: sum} is the situation $k_1 = \cdots = k_l=1$ such that $k=l$. It is an interesting question whether the scoring functions of the form \eqref{eqn: sum} are the only strictly $\F$-consistent scoring functions for $T$, which amounts to the question of \textit{separability} of scoring rules that was posed by \cite{FrongilloKash2014b}. The answer is generally negative. As mentioned in the introduction, it is known that all Bregman functions elicit $T$, if the components of $T$ are all expectations of transformations of $Y$ \citep{Savage1971,OsbandReichelstein1985,DawidSebastian1999,BanerjeeGuoETAL2005,AbernethyFrongillo2012} or ratios of expectations with the same denominator \citep{FrongilloKash2014b}; see also Corollary \ref{cor:Bregman}. However, for other situations, such as a combination of different quantiles and\,/\,or expectiles, the answer is positive; see Corollary \ref{cor:elicitable components}. These results rely on `Osband's principle' which gives necessary conditions for scoring functions to be strictly $\F$-consistent for a given functional $T$; see Section \ref{sec:Osband}.

There are more involved functionals that are $k$-elicitable than just the mere combination of $k$ 1-elicitable components. To illustrate this with a first example, recall that the variance does not have convex level sets in the sense of Proposition \ref{prop:convex level sets}, whence it is not elicitable. However, we can easily show that the pair (expectation, variance) is 2-elicitable.

\begin{cor}\label{cor:expectation-variance}
Let $\F$ be a class of distribution functions on $\R$ with finite second moments. Then, the functional  $T=(T_1,T_2)\colon \F\to\R^2$, defined as $T_1(F) = \int_\R y\diff F(y)$,  $T_2(F) = \int_\R y^2 \diff F(y) - ( \int_\R y\diff F(y))^2$ is 2-elicitable.
\end{cor}

\begin{proof}
Let $\phi\colon \R\to\R$, $z\mapsto \phi(z) = z^2 /(1+|z|)$.
The scoring function $S_1\colon \R\times \R\to \R$, $(x_1, y)\mapsto S_1(x_1,y) = \phi(y) - \phi(x_1) - \phi'(x_1)(y - x_1)$ is a strictly $\F$-consistent scoring function for the expectation and $S_2\colon [0,\infty)\times \R\to\R$, $(x_2,y)\mapsto S_2 (x_2,y) = \phi(y^2) - \phi(x_2) - \phi'(x_2)(y^2 - x_2)$ is a strictly $\F$-consistent scoring function for the second moment. Hence, invoking Lemma \ref{lem: sum}, the pair (expectation, second moment) is 2-elicitable. Using the revelation principle given in Proposition \ref{prop: revelation principle} yields the assertion.
\end{proof}

In Section \ref{sec:spectral}, we show that the concept of $k$-elicitability is not restricted to functionals that can be obtained by combining Lemma \ref{lem: sum} and the revelation principle. It is shown in \citet[Example 3.4]{Weber2006} and \citet[Theorem 11]{Gneiting2011} that the coherent risk measure Expected Shortfall at level $\alpha$, $\alpha\in (0,1)$, does not have convex level sets and is therefore not elicitable. In contrast, we show in Corollary \ref{cor:VaR-ES} that the pair $(\text{\rm Value at Risk}_{\alpha}, \text{\rm Expected Shortfall}_{\alpha})$ is 2-elicitable relative to the class of distributions on $\R$ with finite first moment and unique $\alpha$-quantiles. This refutes Proposition 2.3 of \cite{Osband1985}; see Remark \ref{rem:separability} for a discussion.

\section{Osband's principle}\label{sec:Osband}

In this section, we give necessary conditions for the strict $\F$-consistency of a scoring function $S$ for a functional $T\colon \F\to\A$. In the light of Lemma \ref{lem:consistency} and the discussion thereafter, we have to impose some richness conditions on the class $\F$ as well as on the `variability' of the functional $T$.
To this end, we establish a link between strictly $\F$-consistent scoring functions and strict $\F$-identification functions. We illustrate the idea in the one-dimensional case. Let $\F$ be a class of distribution functions on $\R$, $T\colon \F \to \R$ a functional and $S\colon \R\times \R \to\R$ a strictly $\F$-consistent scoring function for $T$. Furthermore, let $V\colon \R\times \R\to\R$ be an oriented strict $\F$-identification function for $T$. Then, under certain regularity conditions, there is a non-negative function $h\colon \R \to \R$ such that
\be{eqn: Osband naive}
\frac{\dint}{\dint x} S(x,y) = h(x) V(x,y).
\ee
If we na\"ively swap differentiation and expectation and $h$ does not vanish, the form \eqref{eqn: Osband naive} plus the identification property of $V$ are sufficient for the first order condition on $\bar S(\cdot, F)$, $F\in\F$, to be satisfied and the orientation of $V$ as well as the fact that $h$ is positive are sufficient for $\bar S(\cdot, F)$ to satisfy the second order condition for strict $\F$-consistency. So the really interesting part is to show that the form given in \eqref{eqn: Osband naive} is necessary for the strict $\F$-consistency of a scoring function for $T$.

The idea of this characterization originates from \cite{Osband1985}. He gives a characterization including $\R^k$-valued functionals, but for his proof he assumes that $\F$ contains all distributions with finite support. This is not a problem \textit{per se}, but in the light of Lemma \ref{lem:consistency} and the discussion thereafter it would be desirable to weaken this assumption or to complement the result. \cite{Gneiting2011} illustrates Osband's principle in a quite intuitive manner for the one-dimensional case. In \citet[Theorem 5]{SteinwartPasinETAL2014} there is a rigorous statement of Osband's principle for the one-dimensional case. We shall give a proof in the setting of an $\R^k$-valued functional that does not rely on the existence of distributions with finite support in $\F$.

Let $\F$ be a class of distribution functions on $\O\subseteq \R^d$. Fix a functional $T\colon \F \to \A\subseteq\R^k$, an identification function $V\colon \A\times \O\to\R^k$  and a scoring function $S\colon \A\times \O \to \R$.
We introduce the following collection of regularity assumptions.

\setcounter{ass}{0}

\begin{ass}[V]\label{ass:V1}
For every $x\in \interior (\A)$ there are $F_1, \ldots, F_{k+1}\in \F$ such that  
\[
0\in \interior\left( \conv \left(\{ \bar V(x,F_1), \ldots, \bar V(x,F_{k+1}) \}\right)\right).
\]
\end{ass}

\begin{rem}
Assumption (V\ref{ass:V1}) implies that for every $x\in \interior (\A)$ there are $F_1, \ldots, F_k\in \F$ such that the vectors
\(
\bar V(x,F_1), \dots, \bar V(x,F_k)
\)
are linearly independent. 
\end{rem}

Assumption (V\ref{ass:V1}) ensures that the class $\F$ is `rich' enough meaning that the functional $T$ varies sufficiently in order to derive a necessary form of the scoring function $S$ in Theorem \ref{thm: Osband's principle}. We emphasize that assumptions like (V\ref{ass:V1}) are classical in the literature. For the case of $k$-elicitability, \cite{Osband1985} assumes that $0\in \interior \left( \conv \left(\{  V(x,y)\colon y\in \O \}\right)\right)$. 
\citet[Definition 8]{SteinwartPasinETAL2014} and \citet{Lambert2012} treat the case $k=1$ and work under the assumption that the functional is \textit{strictly locally non-constant} which implies assumption (V\ref{ass:V1}) if the functional is identifiable. 

\begin{ass}[V]\label{ass:V0}
For every $F\in \F$, the function 
$\bar V(\cdot,F) \colon \A \to \R^k$, $x\mapsto \bar V(x,F)$,
is continuous.
\end{ass}

\begin{ass}[V]\label{ass:V2}
For every $F\in \F$, the function $\bar V(\cdot, F)$ is continuously differentiable.
\end{ass}

If the function $x\mapsto V(x,y)$, $y\in\O$, is continuous (continuously differentiable), assumption (V\ref{ass:V0}) (assumption (V\ref{ass:V2})) is directly satisfied, and it is even equivalent to (V\ref{ass:V0}) ((V\ref{ass:V2})) if $\F$ contains all measures with finite support. However, (V\ref{ass:V0}) and (V\ref{ass:V2}) are much weaker requirements if we move away from distributions with finite support. To illustrate this fact, let $k=1$ and $V(x,y)= \one\{y\le x\} - \alpha$, $\alpha\in(0,1)$, which is a strict $\F$-identification function for the $\alpha$-quantile. Of course, $V(\cdot, y)$ is not continuous. But if $\F$ contains only probability distributions $F$ that have a continuous derivative $f = F'$, then $\bar V(x,F) = F(x) - \alpha$ and $\frac{\dint}{\dint x} \bar V(x,F) = f(x)$ and $V$ satisfies (V\ref{ass:V0}) and (V\ref{ass:V2}). The following assumptions (S\ref{ass:S1}) and (S\ref{ass:S3}) are similar conditions as (V\ref{ass:V0}) and (V\ref{ass:V2}) but for scoring functions instead of identification functions. 

\setcounter{ass}{0}
\begin{ass}[S]\label{ass:S1}
For every $F\in \F$, the function 
$\bar S(\cdot, F) \colon \A \to \R$, $x\mapsto \bar S(x,F)$,
is continuously differentiable.
\end{ass}

\begin{ass}[S]\label{ass:S3}
For every $F\in \F$, the function $\bar S(\cdot, F)$ is continuously differentiable and the gradient is locally Lipschitz continuous. Furthermore, $\bar S(\cdot, F)$ is twice continuously differentiable at $t = T(F)\in \interior(\A)$. 
\end{ass}

Note that assumption (S\ref{ass:S3}) implies that the gradient of $\bar S(\cdot,F)$ is (totally) differentiable for almost all $x \in \A$ by Rademacher's theorem, which in turn indicates that the Hessian of $\bar S(\cdot,F)$ exists for almost all $x \in \A$ and is symmetric by Schwarz's theorem; see \citet[p. 57]{GrauertFischer1978}.

\begin{thm}[Osband's principle]\label{thm: Osband's principle}
Let $\F$ be a convex class of distribution functions on $\O \subseteq \R^d$. 
Let $T\colon \F \to \A \subseteq \R^k$ be a surjective, elicitable and identifiable functional with a strict $\F$-identification function $V\colon \A \times \O \to \R^k$ and a strictly $\F$-consistent scoring function $S\colon \A\times \O\to \R$. If the assumptions (V\ref{ass:V1}) and (S\ref{ass:S1}) hold, then there exists a matrix-valued function
 $h \colon \interior(\A) \to \R^{k\times k}$ such that for $l \in \{1,\dots,k\}$
\be{eqn: Osband FOC expectation}
\partial_l \bar S(x,F) = \sum_{m=1}^k h_{lm}(x) \bar V_m(x,F)
\ee
for all $x\in \interior(\A)$ and $F\in \F$. If in addition, assumption (V\ref{ass:V0}) holds, then $h$ is continuous. Under the additional assumptions (V\ref{ass:V2}) and (S\ref{ass:S3}), the function $h$ is locally Lipschitz continuous.
\end{thm}

The proof of Theorem \ref{thm: Osband's principle} follows closely the idea of the proof of \citet[Theorem 2.1]{Osband1985}. However, the latter proof only works under the condition that the class $\F$ contains all distributions with finite support. He conjectures that the assertion also holds if $\F$ consists only of absolutely continuous distributions, but we do not believe that his approach is feasible for this case.
To show Theorem \ref{thm: Osband's principle}, we apply a similar technique as in the proof of  \citet[Lemma 2.2]{Osband1985} which is based on a finite-dimensional argument. 

\begin{rem}
Let $\tilde h \colon \A\to \R^{k\times k}$ be a function such that the restriction $\tilde h_{|\interior(\A)} $ to $\interior(\A)$ coincides with the function $h$ in \eqref{eqn: Osband FOC expectation}. Then the function 
\[
{\tilde h}V\colon \A\times \O\to \R^k, \quad (x,y)\mapsto {\tilde h}V(x,y) = \tilde h(x) V(x,y)
\]
is an $\F$-identification function for $T$. If $\det(\tilde h(x))\neq 0$ for all $x\in\A$, then ${\tilde h}V$ is even a strict $\F$-identification function for $T$. However, even if $V$ is oriented, ${\tilde h}V$ is not necessarily an oriented strict $\F$-identification function. 
\end{rem}

Under the conditions of Theorem \ref{thm: Osband's principle}, equation \eqref{eqn: Osband FOC expectation} gives a characterization of the partial derivatives of the expected score. If we impose more smoothness assumptions on the expected score, we are also able to give a characterization of the \textit{second} order derivatives of the expected score. In particular, one has the following result.

\begin{cor}\label{cor: OP SOC}
Let $\F$ be a convex class of distribution functions on $\O \subseteq \R^d$.
For a surjective, elicitable and identifiable functional $T\colon \F\to\A\subseteq\R^k$ with a strict $\F$-identification function $V\colon \A \times \O \to \R^k$ and a strictly $\F$-consistent scoring function $S\colon \A\times \O\to \R$ that satisfy assumptions (V\ref{ass:V1}), (V\ref{ass:V2}) and (S\ref{ass:S3}) we have the following identities for the second order derivatives
\begin{align}\label{eqn: Osband SOC2}
\partial_m \partial_l \bar S(x,F) &= \sum_{i=1}^k  \partial_m h_{li}(x)\bar V_i(x,F) + h_{li}(x)\partial_m \bar V_i(x,F)\\ \nonumber
&= \sum_{i=1}^k  \partial_l h_{mi}(x)\bar V_i(x,F) + h_{mi}(x)\partial_l \bar V_i(x,F) = \partial_l\partial_m \bar S(x,F),
\end{align}
for all $l,m \in\{ 1,\ldots, k\}$, for all $F\in\F$ and almost all $x\in\interior(\A)$, where $h$ is the matrix-valued function appearing at \eqref{eqn: Osband FOC expectation}. In particular, \eqref{eqn: Osband SOC2} holds for $x = T(F) \in \interior(\A)$. 
\end{cor}

Theorem \ref{thm: Osband's principle} and Corollary \ref{cor: OP SOC} establish necessary conditions for strictly $\F$-consistent scoring functions on the level of the expected scores. If the class $\F$ is rich enough and the scoring and identification function smooth enough pointwise in the following sense, we can also deduce a necessary condition for $S$ which holds pointwise.

\setcounter{ass}{0}
\begin{ass}[F]\label{ass:F0}
For every $y\in\O$ there exists a sequence $(F_n)_{n \in \mathbb{N}}$ of distributions $F_n \in \F$ that converges weakly to the Dirac-measure $\delta_y$ such that the support of $F_n$ is contained in a compact set $K$ for all $n$. 
\end{ass}

\setcounter{ass}{0}
\begin{ass}[VS]\label{ass:VS}
Suppose that the complement of the set
\[
C := \{(x,y) \in \A \times \O\;|\; \text{$V(x,\cdot)$ and $S(x,\cdot)$ are continuous at the point $y$}\}
\]
has $(k+d)$-dimensional Lebesgue measure zero.
\end{ass}

\begin{prop}\label{prop: integration}
Let $\F$ be convex. Assume that $\interior(\A)\subseteq\R^k$ is a star domain and let $T\colon \F\to\A$ be a surjective, elicitable and identifiable functional with a strict $\F$-identification function $V\colon \A \times \O \to \R^k$ and a strictly $\F$-consistent scoring function $S\colon \A\times \O\to \R$.  
Suppose that assumptions (V\ref{ass:V1}), (V\ref{ass:V0}), (S\ref{ass:S1}), (F\ref{ass:F0}) and (VS\ref{ass:VS}) hold. Let $h$ be the matrix valued function appearing at \eqref{eqn: Osband FOC expectation}. Then, the scoring function $S$ is necessarily of the form
\begin{align}\label{eq:S}
S(x,y) = \sum_{r=1}^k \sum_{m=1}^k \int_{z_r}^{x_r} h_{rm}&(x_1,\dots,x_{r-1},v,z_{r+1},\dots,z_k) \\\nonumber
& \times V_m(x_1,\dots,x_{r-1},v,z_{r+1},\dots,z_k,y)\diff v + a(y)
\end{align}
for almost all $(x,y) \in \A \times \O$ for some star point $z = (z_1, \ldots, z_k)\in\interior(\A)$ and some $\F$-integrable function $a\colon \O\to\R$. On the level of the expected score $\bar S(x,F)$, equation \eqref{eq:S} holds for all $x \in \interior(A)$, $F \in \hat\F$.
\end{prop}

While Theorem \ref{thm: Osband's principle}, Corollary \ref{cor: OP SOC} and Proposition \ref{prop: integration} only establish necessary conditions for strictly $\F$-consistent scoring functions for some functional $T$, often, they guide a way how to construct strictly $\F$-consistent scoring functions starting with a strict $\F$-identification function $V$ for $T$. 
For the one-dimensional case, one can use the fact that, subject to some mild regularity conditions, if $V$ is a strict $\F$-identification function, then either $V$ or $-V$ is oriented; see Remark \ref{rem: Orientation}. Supposing that $V$ is oriented, we can choose any strictly positive function $h\colon \A\to\R$ to get the derivative of a strictly $\F$-consistent scoring function. Then integration yields the desired strictly $\F$-consistent scoring function.

Establishing \textit{sufficient} conditions for scoring functions to be strictly $\F$-consistent for $T$ is generally more involved in the case $k > 1$. First of all, working under assumption (S\ref{ass:S3}), the symmetry of the Hessian $\nabla^2 \bar S(x,F)$ imposes strong necessary conditions on the functions $h_{lm}$; see for example Proposition \ref{prop:elicitable components} which treats the case where all components of the functional $T = (T_1, \ldots, T_k)$ are elicitable and identifiable. The example of spectral risk measures is treated in Section \ref{sec:spectral}. Secondly, \eqref{eqn: Osband FOC expectation} and \eqref{eqn: Osband SOC2} are necessary conditions for $\bar S(x,F)$ having a \textit{local} minimum in $x = T(F)$, $F\in\F$. Even if we additionally suppose that the Hessian $\nabla^2\bar S(x,F)$ is strictly positive definite at $x=T(F)$, this is a sufficient condition only for a \textit{local} minimum at $x=T(F)$, but does not provide any information concerning a \textit{global} minimum. 
Consequently, even if the functions $h_{lm}$ satisfy \eqref{eqn: Osband SOC2}, one must verify the strict consistency of the scoring function on a case by case basis. This can often be done by showing that the one-dimensional functions $\R\to\R$, $s\mapsto \bar S(t+ sv,F)$, with $t=T(F)$, have a global minimum in $s=0$ for all $v\in \mathbb S^{k-1}$ and for all $F\in\F$. This holds for example if the function $(x,y)\mapsto h(x)V(x,y)$ is an oriented strict $\F$-identification function for $T$; see Lemma \ref{lem:sufficiency}. In this step, one may have to impose additional conditions on the functions $h_{lm}$ to ensure sufficiency which cannot always be shown to be necessary.

We conclude this section with a remark clarifying how the function $h$ in Osband's principle behaves under the revelation principle.
\begin{rem}
Let $g :\A \to \A'$ be a bijection, $\A, \A'\subseteq \R^k$. Suppose we have an identification function $V$ for a functional $T\colon \F\to\A$ and we choose the identification function $V_g(x',y) = V(g^{-1}(x'),y)$ as an identification function for the functional $T_g = g \circ T$. If the functional $T$ (and hence also $T_g$ by Proposition \ref{prop: revelation principle}) is elicitable, then the gradient of the expected scores of $T$ and $T_g$ are of the form \eqref{eqn: Osband FOC expectation} with functions $h$ and $h_g$, respectively. The functions $h$ and $h_g$ are connected by the following relation
\[
(h_g)_{lm}(x') = \sum_{r = 1}^k \partial_l (g^{-1})_r(x') h_{rm}(g^{-1}(x')), \quad x' \in \A'.
\]
\end{rem}

\section{Functionals with elicitable components}\label{sec:elicitable components}

Suppose that the functional $T = (T_1,\dots,T_k)\colon \R \to \A \subseteq \R^k$ consists of $1$-elicitable components $T_m$. As prototypical examples of such $1$-elicitable components, we consider the functionals given in Table \ref{table:identification functions} where we implicitly assume that $\O\subseteq \R$ if a quantile or an expectile are a part of $T$. With the given identification functions, it turns out that usually $T$ (or some subset of its components) fulfills either one of the following two assumptions.

\setcounter{ass}{3}
\begin{ass}[V]\label{ass:V4}
Let assumption (V\ref{ass:V2}) hold. For all $r\in\{1, \ldots, k\}$ and for all $t\in\interior(\A)\cap T(\F)$ there are $F_1, F_2 \in T^{-1}(\{t\})$ such that
\begin{align*}
\partial_l \bar V_l(t,F_1) = \partial_l \bar V_l(t,F_2)\quad \forall l\in\{1,\ldots, k\}\setminus \{r\}, && \partial_r \bar V_r(t,F_1) \neq \partial_r \bar V_r(t,F_2).
\end{align*}
\end{ass}

\begin{ass}[V]\label{ass:V5}
Let assumption (V\ref{ass:V2}) hold. For all $F\in \F$ there is a constant $c_F\neq 0$ such that for all $r\in\{1, \ldots, k\}$ and for all $x\in\interior(\A)$ it holds that
\begin{align*}
\partial_r \bar V_r(x,F) =c_F.
\end{align*}
\end{ass}
Following \citet{FrongilloKash2014b}, we call a functional that fulfills assumption (V\ref{ass:V5}) with $c_F = 1$ for all $F\in\F$ a \emph{linear functional}. 

\begin{table}
\centering
\caption{Strict identification functions for $k=1$; see \citet[Table 9]{Gneiting2011}}\label{table:identification functions}
\begin{tabular}{lll}
\hline\hline
Functional					&& Strict identification function \\
\hline
Ratio $\E_F[p(Y)]/\E_F[q(Y)]$	&& $V(x,y) = xq(y) - p(y)$ \\
$\alpha$-Quantile			&& $V(x,y) = \one\{y\le x\} - \alpha$ \\
$\tau$-Expectile			&& $V(x,y) = 2 |\one\{y\le x\} - \tau |(x-y)$ \\
\hline
\end{tabular}
\end{table}

\textit{Prima facie}, assumptions (V\ref{ass:V4}) and (V\ref{ass:V5}) are mutually exclusive.  Considering the functionals in Table \ref{table:identification functions} with the associated identification functions, we obtain, for $x = (x_1, \ldots, x_k)\in \R^k$, $F\in\F$ with derivative $F'=f$ and $m\in\{1,\ldots, k\}$
\[
\partial_m \bar V_m(x,F) =
\begin{cases}
\bar q_m(F), & \text{if } V_m(x,y) = x_mq_m(y) - p_m(y)\\
f(x_m), &\text{if } V_m(x,y) = \one\{y\le x_m\} - \alpha_m \\
(2-4\tau_m)F(x_m) +2\tau_m, &\text{if } V_m(x,y) = 2 |\one\{y\le x_m\} \\
&\hspace{3cm}- \tau_m |(x_m-y),
\end{cases}
\]
where $p_m,q_m\colon \O\to\R$ are some $\F$-integrable functions such that $\bar q_m(F)\neq 0$ for all $F\in\F$ and $\alpha_m,\tau_m\in(0,1)$. We see that (V\ref{ass:V5}) is satisfied if e.g.~$T$ is a vector of ratios of expectations with \textit{the same denominator} (compare the situation in \citet{FrongilloKash2014b}). In this situation, we have that $c_F = \bar q(F)$. On the other hand, if the components of $T$ are quantiles, expectiles with $\tau_m \not=1/2$ or ratios of expectations with \textit{different denominators} and additionally the class $\F$ is rich enough, then (V\ref{ass:V4}) might be satisfied.

\begin{prop}\label{prop:elicitable components} Let $T_m\colon \F\to \A_m\subseteq \R$ be $1$-elicitable and $1$-identifiable functionals with oriented strict $\F$-identification functions $V_m\colon \A_m\times \O\to\R$ for $m\in\{1, \ldots, k\}$. 
Let $\A:=T(\F) \subseteq \A_1\times \cdots \times \A_k$. Then $V\colon \A \times \O \to \R^k$ defined as 
\be{eq:V1}
V(x_1,\dots,x_k,y) = \big(V_1(x_1,y),\dots,V_k(x_k,y)\big)^{\top}
\ee
is an oriented strict $\F$-identification function for $T = (T_1, \ldots, T_k)$.

Let $\F$ be convex and $S\colon \A\times \O \to \R$ be a strictly $\F$-consistent scoring function for $T = (T_1, \ldots, T_k)$. Suppose that assumptions (V\ref{ass:V1}), (V\ref{ass:V2}) and (S\ref{ass:S3}) hold, and let $h:\interior(\A) \to \R^{k\times k}$ be the function given at \eqref{eqn: Osband FOC expectation}. Define $\A_m':= \{x_m \colon \exists (z_1,\dots,z_k)\in \interior(A), z_m=x_m \}$.
\begin{enumerate}
\item[(i)]
If assumption (V\ref{ass:V4}) holds and $\A$ is connected then there are functions $g_m\colon \A_m'\to \R$, $m\in\{1, \ldots, k\}$, $g_m > 0$, such that 
\[
h_{mm}(x_1, \ldots, x_k) = g_m(x_m) 
\]
for all $m\in\{1, \ldots, k\}$ and $(x_1, \ldots, x_k)\in\interior(\A)$ and 
\be{eq:h_rl}
h_{rl}(x)=0
\ee
for all $r,l\in\{1,\ldots, k\}$, $l\neq r$, and for all $x\in\interior(\A)$.
\item[(ii)]
If assumption (V\ref{ass:V5}) holds then 
\begin{align}\label{eq:cyclic}
\partial_l h_{rm}(x) = \partial_r h_{lm}(x), && h_{rl}(x) = h_{lr}(x)
\end{align}
for all $r,l,m\in\{1,\ldots, k\}$, $l\neq r$, where the first identity holds for almost all $x\in \interior(\A)$ and the second identity for all $x\in\interior(\A)$.
Moreover, the matrix $\big(h_{rl}(x)\big)_{l,r=1, \ldots, k}$ is positive definite for all $x\in\interior(\A)$.
\end{enumerate}
\end{prop}

A direct consequence of Proposition \ref{prop:elicitable components} (i) and Proposition \ref{prop: integration} is the following characterization of the class of strictly $\F$-consistent scoring functions for functionals with elicitable components satisfying assumption (V\ref{ass:V4}). In particular, it gives a characterization of the class of strictly $\F$-consistent scoring functions for a vector of different quantiles and\,/\,or different expectiles (with the exception of the $1/2$-expectile), thus answering a question raised in \citet[p.\,370]{GneitingRaftery2007}. 

\begin{cor}\label{cor:elicitable components}
Let $\F$ be convex.
Suppose that $T = (T_1, \ldots, T_k)\colon \F\to\A$ is a functional with $1$-identifiable components having oriented strict $\F$-identification functions. Assume that the interior of $\A := T(\F) \subseteq \A_1\times \cdots \times \A_k$ is a star domain and that assumptions (V\ref{ass:V1}), (V\ref{ass:V2}), (S\ref{ass:S3}), (F\ref{ass:F0}) and (VS\ref{ass:VS}) hold for $T$. If assumption (V\ref{ass:V4}) holds, then a scoring function $S\colon \A\times \O\to\R$ is strictly $\F$-consistent for $T$ if and only if it is of the form
\be{eq:S sum}
S(x_1, \ldots, x_k,y) = \sum_{m=1}^k S_m(x_m,y),
\ee
for almost all $(x,y) \in \A \times \O$, where $S_m\colon\A_m\times \O\to\R$, $m\in\{1, \ldots, k\}$, are some strictly $\F$-consistent scoring functions for $T_m$.
\end{cor}

If we are in the situation of Proposition \ref{prop:elicitable components} (ii), that is, $T$ satisfies assumption (V\ref{ass:V5}), it is well-known that a statement analogous to Corollary \ref{cor:elicitable components} is false. Let $F \in \F$ and $t = T(F)$. 
Recalling the orientation of the components $V_m$, we can immediately deduce that there is $c_F>0$ such that $\bar V(t+sv,F) = c_F sv$ for $s\in\R$ and $v\in \mathbb S^{k-1}$. Hence, one obtains
\[
v^\top h(t+sv) \bar V(t+sv,F) = c_Fsv^\top h(t+sv)v.
\]
Consequently, if $\A$ is open and convex, the positive definiteness of $h(x)$ for all $x\in\A$ is a sufficient condition for the strict $\F$-consistency of $S$ for $T$ by Lemma \ref{lem:sufficiency} (i). Moreover, we now assume that $T$ is a ratio of expectations with the same denominator $q\colon \O\to \R$ implying that $c_F = \bar q(F)$ for all $\F\in\F$. Using Proposition \ref{prop: integration} and partial integration, we obtain that for almost all $(x,y) \in \A \times \O$ strictly $\F$-consistent scoring functions for $T$ are of the form
\begin{equation}\label{eq:Bregman}
S(x,y) = -\phi(x)q(y) + \sum_{m=1}^k V_m(x,y)\partial_m \phi(x) + a(y), 
\end{equation}
with
\begin{equation}\label{eq:Bregman phi}
\phi(x) = \sum_{r = 1}^k \int_{z_r}^{x_r} \int_{z_r}^v h_{rr}(x_1,\dots,x_{r-1},w,z_{r+1},\dots,z_k)\dint w \dint v,
\end{equation}
where $(z_1, \ldots, z_k)\in \A$ and $a\colon\O\to\R$ is some $\F$-integrable function.
Using \eqref{eq:cyclic}, it follows that the function $\phi$ has Hessian $h$. Therefore, for $\A$ open and convex, $\phi$ is strictly convex. Hence we have shown the following corollary.

\begin{cor}\label{cor:Bregman}
Let $\F$ be convex.
Let $T = (T_1, \ldots, T_k)\colon \F\to\A \subseteq \R^k$ be a ratio of expectations with the same denominator $q\colon \O\to\R$, $q>0$. More specifically, let $T$ be a surjective functional with $1$-identifiable components with oriented strict identification functions $V_m:\A_m \times \O \to \R$, $m\in\{1, \ldots, k\}$, that fulfills assumption (V\ref{ass:V5}). 
Suppose that  $\A \subseteq \A_1\times \cdots \times \A_k$ is open and convex and that assumptions (V\ref{ass:V1}), (V\ref{ass:V2}), (S\ref{ass:S3}), (F\ref{ass:F0}) and (VS\ref{ass:VS}) hold. Then, a scoring function $S$ is strictly $\F$-consistent for $T$ if and only if it is of the form \eqref{eq:Bregman} for almost all $(x,y) \in \A \times \O$ with a twice continuously differentiable strictly convex function $\phi:\A \to \R$ of the form \eqref{eq:Bregman phi} and an $\F$-integrable function $a\colon \O \to \R$.
\end{cor}

This corollary recovers results of \citet{OsbandReichelstein1985,BanerjeeGuoETAL2005,AbernethyFrongillo2012} if $T$ is linear (meaning $q\equiv 1$), which show that all consistent scoring functions for linear functionals are so-called \emph{Bregman functions}, that is, functions of the form \eqref{eq:Bregman} with $q \equiv 1$ and a convex function $\phi$. \citet[Theorem 13]{FrongilloKash2014b} also treat the case of more general functions $q$. Comparing these results with Corollary \ref{cor:Bregman}, one can see that on the one hand, they are stronger as they require weaker smoothness assumptions on the scoring function, but on the other hand, they are weaker since they assume that $\F$ contains all one-point distributions $\delta_y$.

\begin{rem}\label{rem:elicitable components}
One might wonder about necessary conditions on the matrix-valued function $h$ in the flavor of Proposition \ref{prop:elicitable components} if the $k$ components of the functional $T$ can be regrouped into (i) a new functional $T'_1\colon \F\to\A'_1\subset \R^{k'_1}$ with an oriented strict $\F$-identification function $V'_1\colon \A'_1\times \O\to \R^{k'_1}$ which satisfies assumption (V\ref{ass:V4}), and (ii) several, say $l$, new functionals $T'_m\colon \F\to\A'_{k'_m}\subseteq \R^{k'_m}$, $m\in\{2,,\ldots, l+1\}$ with oriented strict $\F$-identification functions $V'_m\colon \A'_m\times \O\to \R^{k'_m}$ such that each one satisfies assumption (V\ref{ass:V5}), and $k'_1+ \cdots + k'_{l+1} = k$. We can apply Proposition \ref{prop:elicitable components} to obtain necessary conditions for each of the $(k'_m\times k'_m)$-valued functions $h'_m$, $m\in\{1,\ldots, l+1\}$.
Applying Lemma \ref{lem: sum} we get a possible choice for a strictly $\F$-consistent scoring function $S$ for $T$. On the level of the $k\times k$-valued function $h$ associated to $S$ this means that $h$ is a block diagonal matrix of the form $\text{diag}(h'_1, \ldots, h'_{l+1})$. But what about the necessity of this form?
Indeed, if we assume that the blocks in (ii) have maximal size (or equivalently that $l$ is minimal) then one can verify that $h$ must be necessarily of the block diagonal form described above.
\end{rem}


\section{Spectral risk measures}\label{sec:spectral}

Risk measures are a common tool to measure the risk of a financial position $Y$. A risk measure is usually defined as a mapping $\rho$ from some space of random variables, for example $L^{\infty}$, to the real line. Arguably, the most common risk measure in practice is \emph{Value at Risk} at level $\alpha$ ($\VaR_\alpha$) which is the generalized $\alpha$-quantile $F^{-1}(\alpha)$, that is,
\[
\VaR_{\alpha}(Y) := F^{-1}(\alpha) := \inf\{x \in \R \colon F(x) \ge \alpha\},
\] 
where $F$ is the distribution function of $Y$. 
An important alternative to $\VaR_\alpha$ is \emph{Expected Shortfall} at level $\alpha$ ($\ES_\alpha$) (also known under the names Conditional Value at Risk or Average Value at Risk). It is defined as 
\begin{equation}\label{eq:ES}
\ES_\alpha(Y):= \frac{1}{\alpha} \int_0^\alpha \text{VaR}_u(Y) \diff u, \quad \alpha\in(0,1],
\end{equation}
and $\ES_0(Y) = \essinf Y$. Since the influencial paper of \citet{ArtznerDelbaenETAL1999} introducing coherent risk measures, there has been a lively debate about which risk measure is best in practice, one of the requirements under discussion being the \text{coherence} of a risk measure. We call a functional $\rho$ \emph{coherent} if it is monotone, meaning that $Y\le X$ a.s.~implies that $\rho(Y)\le \rho(X)$; it is superadditive in the sense that $\rho(X+Y)\ge \rho(X) + \rho(Y)$; it is positively homogeneous which means that $\rho(\lambda Y) = \lambda \rho(Y)$ for all $\lambda \ge0$; and it is translation invariant which amounts to $\rho(Y+a) = \rho(Y)+a$ for all $a\in \R$. In the literature on risk measures there are different sign conventions which co-exist. In this paper, a positive value of $Y$ denotes a profit. Moreover, the position $Y$ is considered the more risky the smaller $\rho(Y)$ is. Strictly speaking, we have chosen to work with utility functions instead of risk measures as for example in \citet{Delbaen2012}. The risk measure $\rho$ is called \emph{comonotonically additive} if $\rho(X+Y) = \rho(X) + \rho(Y)$ for comonotone random variables $X$ and $Y$. Coherent and comonotonically additive risk measures are also called \emph{spectral} risk measures \citep{Acerbi2002}. All risk measures of practical interest are law-invariant, that is, if two random variables $X$ and $Y$ have the same law $F$, then $\rho(X) = \rho(Y)$. As we are only concerned with law-invariant risk measures in this paper, we will abuse notation and write $\rho(F) := \rho(X)$, if $X$ has distribution $F$.

One of the main criticisms on $\VaR_\alpha$ is its failure to fulfill the superadditivity property in general \citep{Acerbi2002}. Furthermore, it fails to take the size of losses beyond the level $\alpha$ into account \citep{DanielssoEmbrechtsETAL2001}. In both of these aspects, $\ES_\alpha$ is a better alternative as it is coherent and comonotonically additive, that is, a spectral risk measure. However, with respect to robustness, some authors argue that $\VaR_\alpha$ should be preferred over $\ES_\alpha$ \citep{ContDeguestETAL2010,KouPengETAL2013}, whereas others argue that the classical statistical notions of robustness are not necessarily appropriate in a risk measurement context \citep{KratschmeSchiedETAL2012,KratschmeSchiedETAL2013a,KratschmeSchiedETAL2013}. Finally, $\ES_\alpha$ fails to be $1$-elicitable \citep{Weber2006,Gneiting2011}, whereas $\VaR_\alpha$ is $1$-elicitable for most classes of distributions $\F$ of practial relevance. In fact, except for the expectation, all spectral risk measures fail to be $1$-elicitable \citep{Ziegel2014}; further recent results on elicitable risk measures include \citep{KouPeng2014,WangZiegel2014} showing that distortion risk measures are rarely elicitable and \citep{Weber2006,BelliniBignozzi2014,DelbaenBelliniETAL2014} demonstrating that convex risk measures are only elicitable if they are shortfall risk measures.

We show in Theorem \ref{thm:spectral} (see also Corollary \ref{cor:spectral} and \ref{cor:VaR-ES}) that spectral risk measures having a spectral measure with finite support can be a component of a $k$-elicitable functional. In particular, the pair $(\text{VaR}_\alpha,\,\text{ES}_\alpha)\colon \F\to\R^2$ is 2-elicitable for any $\alpha\in(0,1)$ subject to mild conditions on the class $\F$. We remark that our results substantially generalize the result of \citet{AcerbiSzekely2014} as detailed below.

\begin{defn}[Spectral risk measures]\label{defn:spectral}
Let $\mu$ be a probability measure on $[0,1]$ (called \textit{spectral measure}) and let $\F$ be a class of distribution functions on $\R$ with finite first moments. Then, the \textit{spectral risk measure} associated to $\mu$ is the functional $\nu_\mu\colon \F \to \R$ defined as 
\begin{align*}
\nu_\mu(F):= \int_{[0,1]} \ES_{\alpha}(F)\mu(\dint \alpha).
\end{align*}
\end{defn}

\citet{Kusuoka2001,JouiniSchachermETAL2006} have shown that law-invariant coherent and comonotonically additive risk measures are exactly the spectral risk measures in the sense of Definition \ref{defn:spectral} for distributions with compact support. If $\mu = \delta_{\alpha}$ for some $\alpha\in[0,1]$, then $\nu_\mu(F) = \text{ES}_\alpha(F)$. In particular, $\nu_{\delta_1}(F) =  \int y\diff F(y)$ is the expectation of $F$.

In the following theorem, we show that spectral risk measures whose spectral measure $\mu$ has finite support in $(0,1)$ are $k$-elicitable for some $k$. It is possible to extend the result to spectral measures with finite support in $(0,1]$; see Corollary \ref{cor:spectral}. If $\mu$ has mass at zero, we believe that $\nu_\mu$ is not $k$-elicitable for any $k$ with respect to interesting classes $\F$. In this case, if the support of $F$ is unbounded below, we have $\nu_\mu(F) = \essinf(F) = -\infty$. 

\begin{thm}\label{thm:spectral}
Let $\F$ be a class of distribution functions on $\R$ with finite first moments. 
Let $\nu_\mu\colon \F\to\R$ be a spectral risk measure where $\mu$ is given by
\[
\mu = \sum_{m=1}^{k-1}p_m \delta_{q_m},
\]
with $p_m\in(0,1]$, $\sum_{m=1}^{k-1}p_m=1$, $q_m\in(0,1)$ and the $q_m$'s are pairwise distinct. Define the functional
$T = (T_1, \ldots, T_k)\colon \F \to \R^k$,
where $T_m(F):= F^{-1}(q_m)$, $m \in\{1, \ldots, k-1\}$, and $T_k(F):= \nu_\mu(F)$. Then the following assertions are true:
\begin{enumerate}
\item[(i)] 
If the distributions in $\F$ have unique $q_m$-quantiles, $m\in\{1,\ldots, k-1\}$, then the functional $T$ is $k$-elicitable with respect to $\F$. 

\item[(ii)]
Let $\A\supseteq T(\F)$ be convex and set $\A_r':=\{x_r\colon \exists(z_1, \ldots, z_k)\in\A, x_r=z_r\}$, $r\in\{1, \ldots, k\}$. Define the scoring function $S\colon \A \times \R\to\R$ by
\begin{align}\label{eq:S spectral}
S(x,y) &=
 \sum_{r=1}^{k-1} \big(\one\{y\le x_r\} - q_r \big)G_r(x_r) - \one\{y\le x_r\}G_r(y)\\ \nonumber
&+G_k(x_k) \left(x_k + \sum_{m=1}^{k-1} \frac{p_m}{q_m} \big(\one\{y\le x_m\}(x_m -y) - q_mx_m \big)\right) \\ 
& - \mathcal G_k(x_k) + a(y),\nonumber
\end{align}
where $a\colon \R\to\R$ is $\F$-integrable, $G_r\colon \A_r'\to\R$, $r\in\{1, \ldots, k\}$, $\mathcal G_k\colon \A_k'\to\R$ with $\mathcal G_k'=G_k$ and for all $r\in\{1, \ldots, k\}$ and all $x_r\in\A_r'$ the functions $ \one_{(\infty,x_r]}G_r$ are $\F$-integrable.

If $\mathcal G_k$ is convex and for all $r\in\{1, \ldots, k-1\}$ and $x_k \in \A_k'$, the function
\begin{align}\label{eq:increasing}
\A'_{r,x_k} \to \R, \quad x_r\mapsto\ x_r \frac{p_r}{q_r} G_k(x_k) + G_r(x_r)
\end{align}
with $\A'_{r,x_k}:=\{x_r : \exists (z_1,\dots,z_{k}) \in \A, x_r =z_r, x_k=z_k\}$ is increasing, then $S$ is $\F$-consistent for $T$. If additionally the distributions in $\F$ have unique $q_m$-quantiles, $m \in \{1,\dots,k-1\}$, $\mathcal G_k$ is strictly convex and the functions given at \eqref{eq:increasing} are strictly increasing, then $S$ is strictly $\F$-consistent for $T$.

\item[(iii)]
Assume that the elements of $\F$ have unique $q_m$-quantiles, $m \in \{1,\dots,k-1\}$ and continuous densities. 
Define the function $V\colon \A\times \R \to \R^k$ with components
\begin{align}\label{eq:V spectral}
\begin{split}
V_m(x_1, \ldots, x_k,y) &= \one\{y\le x_m\} - q_m, \ m\in\{1, \ldots, k-1\},\\
V_k(x_1, \ldots, x_k,y) &= x_k - \sum_{m=1}^{k-1} \frac{p_m}{q_m}\,  y\, \one\{y\le x_m\}.
\end{split}
\end{align}
Then $V$ is a strict $\F$-identification function for $T$ satisfying assumption (V\ref{ass:V2}). 

If additionally $\F$ is convex, the interior of $\A:=T(\F)\subseteq\R^k$ is a star domain, (V\ref{ass:V1}) and (F\ref{ass:F0}) hold, and $(V_1,\ldots, V_{k-1})$ satisfies (V\ref{ass:V4}),
then every strictly $\F$-consistent scoring function $S\colon \A\times \R\to\R$ for $T$ satisfying (S\ref{ass:S3}), (VS\ref{ass:VS}) is necessarily of the form given at \eqref{eq:S spectral} almost everywhere. Additionally, $\mathcal G_k$ must be strictly convex and the functions at \eqref{eq:increasing} must be strictly increasing.
\end{enumerate}
\end{thm}

\begin{rem}\label{rem:separability}
According to Theorem \ref{thm:spectral}, the pair $(\VaR_\alpha(F),\ES_\alpha(F))$, and more generally $(F^{-1}(q_1),\dots,F^{-1}(q_{k-1}),\nu_\mu(F))$, admits only non-se\-par\-able strictly consistent scoring functions. This result gives an example demonstrating that \citet[Proposition 2.3]{Osband1985} cannot be correct as it states that any strictly consistent scoring function for a functional with a quantile as a component must be separable in the sense that it must be the sum of a strictly consistent scoring function for the quantile and a strictly consistent scoring function for the rest of the functional.
\end{rem}

Using Theorem \ref{thm:spectral} and the revelation principle (Proposition \ref{prop: revelation principle}) we can now state one of the main results of this paper.

\begin{cor}\label{cor:spectral}
Let $\F$ be a class of distribution functions on $\R$ with finite first moments and unique quantiles.
Let $\nu_\mu\colon \F\to\R$ be a spectral risk measure. If the support of $\mu$ is finite with $L$ elements and contained in $(0,1]$, then $\nu_\mu$ is a component of a $k$-elicitable functional where
\begin{enumerate}
\item[(i)]
$k=1$, if $\mu$ is concentrated at 1 meaning $\mu(\{1\})=1$;
\item[(ii)]
$k= 1+L$, if $\mu(\{1\})<1$.
\end{enumerate}
\end{cor}

In the special case of $T = (\VaR_\alpha, \ES_\alpha)$, the maximal sensible action domain is $\A_0 = \{x \in \mathbb{R}^2 : x_1 \ge x_2\}$ as we always have $\ES_\alpha(F) \le \VaR_{\alpha}(F)$. For this action domain, the characterization of consistent scoring functions of Theorem \ref{thm:spectral} simplifies as follows.

\begin{cor}\label{cor:VaR-ES}
Let $\alpha \in (0,1)$. Let $\F$ be a class of distribution functions on $\R$ with finite first moments and unique $\alpha$-quantiles. Let $\A_0 = \{x \in \mathbb{R}^2 : x_1 \ge x_2\}$. A scoring function $S\colon \A_0\times \R\to\R$ of the form
\begin{align}\label{eq:S VaR,ES}
S(x_1,x_2,y) &= \big(\one\{y\le x_1\} - \alpha\big)G_1(x_1) - \one\{y\le x_1\}G_1(y) \\ \nonumber
&\quad + G_2(x_2)\left(x_2 - x_1 + \frac{1}{\alpha} \one\{y\le x_1\}(x_1-y) \right) \\ \nonumber
&\quad - \mathcal G_2(x_2)+a(y),
\end{align}
where $G_1, G_2, \mathcal G_2, a\colon \R\to\R$, $\mathcal G_2'=G_2$, $a$ is $\F$-integrable and $\one_{(-\infty, x_1]}G_1$ is $\F$-integrable for all $x_1\in\R$,
is $\F$-consistent for $T$ if $G_1$ is increasing and $\mathcal G_2$ is increasing and convex. If $\mathcal G_2$ is strictly increasing and strictly convex, then $S$ is strictly $\F$-consistent for $T$.

Under the conditions of Theorem \ref{thm:spectral} (iii) all strictly $\F$-consistent scoring functions for $T$ are of the form \eqref{eq:S VaR,ES} almost everywhere.
\end{cor}


\cite{AcerbiSzekely2014} also give an example of a scoring function for the pair $T = (\text{\rm VaR}_{\alpha}, \text{\rm ES}_{\alpha})\colon \F\to \A\subseteq \R^2$. They use a different sign convention for $\text{\rm VaR}_{\alpha}$ and $\text{\rm ES}_{\alpha}$ than we do in this paper. Using our sign convention, their proposed scoring function $S^W\colon \A\times \R\to\R$ reads
\begin{align}\label{eq:S^W}
S^W(x_1,x_2,y) &= \alpha\big( x_2^2/2 + W x_1^2/2 - x_1 x_2\big) \\ \nonumber
&+ \one\{y\le x_1\} \big(-x_2(y - x_1) +  W(y^2 - x_1^2)/2 \big),
\end{align}
where $W\in\R$. The authors claim that $S^W$ is a strictly $\F$-consistent scoring function for $T=(\text{\rm VaR}_{\alpha}, \text{\rm ES}_{\alpha})$ provided that 
\be{eq:assumption Acerbi}
\text{\rm ES}_{\alpha}(F) > W\, \text{\rm VaR}_{\alpha}(F)
\ee
for all $F\in\F$. This means that they consider a strictly smaller action domain than $\A_0$ in Corollary \ref{cor:VaR-ES}. They assume that the distributions in $\F$ have continuous densities, unique $\alpha$ quantiles, and that $F(x) \in (0,1)$ implies $f(x) > 0$ for all $F \in F$ with density $f$. Furthermore, in order to ensure that $\bar S^W(\cdot,F)$ is finite one needs to impose the assumption that $\int_{-\infty}^x y^2 \dint F(y)$ is finite for all $x\in \R$ and $F \in \F$. This is slightly less than requiring finite second moments. As a matter of fact, they only show that $\nabla \bar S^W(t_1,t_2,F) = 0$ for $F\in\F$ and $(t_1,t_2) = T(F)$ and that $\nabla^2 \bar S^W (t_1,t_2,F)$ is positive definite. This only shows that $\bar S^W(x,F)$ has a \textit{local} minimum at $x = T(F)$ but does not provide a proof concerning a \textit{global} minimum; see also the discussion after Corollary \ref{cor: OP SOC}. 
However, we can use Theorem \ref{thm:spectral} (ii) to verify their claims with $G_1(x_1) = -(W/2)x_1^2$, $\mathcal G_2(x_2) = (\alpha/2) x_2^2$ and $a=0$. Hence, $\mathcal G_2$ is strictly convex, and the function $x_1\mapsto x_1G_2(x_2)/\alpha +G_1(x_1)$ is strictly increasing in $x_1$ if and only if $x_2>Wx_1$ as at \eqref{eq:assumption Acerbi}. 

The scoring function $S^W$ has one property which is potentially relevant in applications. If $x_1,x_2$ and $y$ are expressed in the same units of measurement, then $S^W(x_1,x_2,y)$ is a quantity with these units squared. If one insists that we should only add quantities with the same units, then the necessary condition that $x_1\mapsto x_1G_2(x_2)/\alpha +G_1(x_1)$ is strictly increasing enforces a condition of the type \eqref{eq:assumption Acerbi}. The action domain is restricted for $S^W$ and the choice of $W$ may not be obvious in practice. Similarly, for the maximal action domain $\A_0$, an open question of practical interest is the choice of the functions $G_1$ and $\mathcal G_2$ in \eqref{eq:S VaR,ES}. We would like to remark that $S$ remains stricly consistent upon choosing $G_1 = 0$ and $\mathcal G_2$ stricly increasing and strictly convex.

\section{Discussion}\label{sec:discussion}

We have investigated necessary and sufficient conditions for the elicitability of $k$-di\-men\-sio\-nal functionals of $d$-di\-men\-sio\-nal distributions. In order to derive necessary conditions we have adapted Osband's principle for the case where the class $\F$ of distributions does not necessarily contain distributions with finite support. This comes at the cost of certain smoothness assumptions on the expected scores $\bar S(\cdot,F)$. For particular situations, e.g.~when characterizing the class of strictly $\F$-consistent scoring functions for ratios of expectations, it is possible to weaken the smoothness assumptions; see \cite{FrongilloKash2014b}. However, \cite{FrongilloKash2014b} assume that the
class $\F$ of distributions contains all distributions with finite support, which is not necessary for the validity of our result. While this is not a great gain in the case of linear functionals or ratios of expectations it comes in handy when considering spectral risk measures. Value at Risk, $\VaR_\alpha$, being defined as the smallest $\alpha$-quantile, is generally not elicitable for distributions where the $\alpha$-quantile is not unique. Therefore, we believe that it is also not possible to show joint elicitability of $(\VaR_\alpha,\ES_\alpha)$ for classes $\F$ of distributions with non-unique $\alpha$-quantiles. However, we can give at least consistent scoring functions which become strictly consistent as soon as the elements of $\F$ have unique quantiles. Fortunately, the classes $\F$ of distributions that are relevant in risk management usually consist of absolutely continuous distributions having unique quantiles. 

\citet{EmmerKratzETAL2013} have remarked that $\ES_\alpha$ is \emph{conditionally elicitable}. One can slightly generalize their definition of conditional elicitability as follows. 
\begin{defn}\label{defn:conditional elicitability}
Fix an integer $k\ge1$. A functional $T_k \colon \F \to \A_k \subseteq \R$ is called \emph{conditionally elicitable of order $k$} if there are $k-1$ elicitable functionals $T_m\colon \F\to\A_m\subseteq\R$, $m\in\{1, \ldots, k-1\}$, such that $T_k$ is elicitable restricted to the class
\[
\F_{x_1,\dots,x_{k-1}} := \{F \in \F \colon T_1(F) = x_1,\dots,T_{k-1}(F) = x_{k-1}\}
\]
for any $(x_1,\dots,x_{k-1})  \in \A_1 \times \cdots \times \A_{k-1}$.
\end{defn}
\emph{Mutatis mutandis}, one can define a notion of \emph{conditional identifiability} by replacing the term `elicitable' with `identifiable' in the above definition. It is not difficult to check that any conditionally identifiable functional $T_k$ of order $k$ is a component of an identifiable functional $T = (T_1,\dots,T_k)$. Spectral risk measures $\nu_\mu$ with spectral measure $\mu$ with finite support in $(0,1)$ provide an example of a conditionally elicitable functional of order $L+1$, where $L$ is the cardinality of the support of $\mu$; see Theorem \ref{thm:spectral}. However, we would like to stress that it is generally an open question whether any conditionally elicitable and identifiable functional $T_k$ of order $k\ge2$ is always a component of a $k$-elicitable functional. 

Slightly modifying \citet[Definition 11]{Lambert2008}, one could define the \emph{elicitability order} of a real-valued functional $T$ as the smallest number $k$ such that the functional is a component of a $k$-elicitable functional. It is clear that the elicitability order of the variance is two, and we have shown that the same is true for $\ES_\alpha$ for reasonably large classes $\F$. For spectral risk measures $\nu_\mu$, the elicitability order is at most $L+1$, where $L$ is the cardinality of the support; see Theorem \ref{thm:spectral}. 

In the one-dimensional case, \cite{SteinwartPasinETAL2014} have shown that having convex level sets in the sense of Proposition \ref{prop:convex level sets} is a sufficient condition for elicitability of a functional $T$ under continuity assumptions on $T$. Without such continuity assumptions, the converse of Proposition \ref{prop:convex level sets} is generally false; see \cite{Heinrich2013} for the example of the mode functional. It is an open (and potentially difficult) question under which conditions a converse of Proposition \ref{prop:convex level sets} is true for higher order elicitability.

\section{Proofs}\label{sec:proofs}

\subsubsection*{Proof of Lemma \ref{lem:sufficiency}}

The first part is a direct consequence of the definition of strict $\F$-consistency.
For the second part, we use part (i) and consider $\psi \colon D\to\R$, $s\mapsto \bar S(t+sv,F)$ for $t = T(F)\in\interior(\A)$, $v\in\mathbb S^{k-1}$ and $D = \{s\in\R\colon t+sv\in\interior(\A)\}$. The strict orientation of $\nabla S$ implies that $\psi'(s) = v^\top \nabla \bar S(t+sv,F) = 0$ if $s=0$, $\psi'(s) > 0$ for $s > 0$ and $\psi'(s) < 0$ for $s < 0$.
\qed

\subsubsection*{Proof of Theorem \ref{thm: Osband's principle}}

Let $x\in \interior(\A)$.
The identifiability property of $V$ plus the first order condition stemming from the strict $\F$-consistency of $S$ yields the relation $\bar V(x,F)=0 \implies \nabla \bar S(x,F)=0$ for all $F\in\F$. Let $l \in\{1,\ldots, k\}$. To show \eqref{eqn: Osband FOC expectation},
consider the composed functional
\[
\bar B(x,\cdot)\colon \F \to \R^{k+1}, \quad F\mapsto (\partial_l \bar S(x,F), \bar V(x,F)).
\]
By construction, we know that 
\be{eqn: kernel}
\bar V(x,F)=0 \ \Longleftrightarrow \ \bar B(x,F)=0
\ee
for all $F\in\F$.
Assumption (V\ref{ass:V1}) implies that there are $F_1, \ldots, F_{k+1}\in \F$ such that the matrix
\[
\mathbb V = \mat\big(\bar V(x,F_1), \ldots, \bar V(x,F_{k+1}) \big) \in \R^{k\times (k+1)}
\]
has maximal rank, meaning $\rank(\mathbb V) = k$. If $\rank(\mathbb V) < k$, then $\text{span}\{\bar V(x,F_1), \ldots, \bar V(x,F_{k+1}) \}$ would be a linear subspace such that the interior of $\conv(\{\bar V(x,F_1), \ldots, \bar V(x,F_{k+1}) \})$ would be empty.
Let $G\in\F$. Then still
\(
0\in \interior( \conv(\{\bar V(x,G), \bar V(x,F_1), \ldots, \bar V(x,F_{k+1})\})),
\)
such that $\rank(\mathbb V_G) = k$ where
\[
\mathbb V_G = \mat\big(\bar V(x,G),\bar V(x,F_1), \cdots, \bar V(x,F_{k+1}) \big) \in \R^{k\times (k+2)}.
\]
Define the matrix
\[
\mathbb B_G = \begin{pmatrix}
\partial_l \bar S(x,G) & \partial_l \bar S(x,F_1) & \cdots & \partial_l \bar S(x,F_{k+1}) \\
 & \mathbb V_G 
\end{pmatrix}
\in \R^{(k+1)\times (k+2)}.
\]
We use \eqref{eqn: kernel} to show that $\ker(\mathbb B_G) = \ker(\mathbb V_G)$. 
First observe that the relation $\ker(\mathbb B_G) \subseteq \ker(\mathbb V_G)$ is clear by construction. To show the other inclusion, let $\theta\in\ker(\mathbb V_G)$ be an element of the simplex. Then \eqref{eqn: kernel} and the convexity of $\F$ yields that $\theta \in \ker(\mathbb B_G)$. By linearity, the inclusion holds also for all $\theta\in\ker(\mathbb V_G)$ with nonnegative components. Finally, let $\theta\in \ker(\mathbb V_G)$ be arbitrary. Assumption (V\ref{ass:V1}) implies that there is $\theta^*\in\ker(\mathbb V_G)$ with strictly positive components. Hence, there is an $\varepsilon >0$ such that $\theta^* + \varepsilon \theta$ has nonnegative components. Since $\mathbb V_G(\theta^* + \varepsilon \theta) = \mathbb V_G \theta^* + \varepsilon \mathbb V_G\theta=0$, we know that $\theta^* + \varepsilon \theta\in \ker(\mathbb B_G)$. Again using linearity and the fact that $\theta^*\in\ker(\mathbb B_G)$ we obtain that $\theta \in \ker(\mathbb B_G)$.

With the rank-nullity theorem, this gives $\rank(\mathbb B_G) = \rank(\mathbb V_G) = k$. Hence, there is a unique vector $(h_{l1}(x),\dots,h_{lk}(x)) \in \R^k$ such that 
\[
\partial_l \bar S(x,G) = \sum_{m=1}^k h_{lm}(x) \bar V_m(x,G).
\]
Since $G\in \F$ was arbitrary, the assertion at \eqref{eqn: Osband FOC expectation} follows.

The second part of the claim can be seen as follows. For $x \in \interior(\A)$ pick $F_1,\dots,F_k \in \F$ such that $\bar V(x,F_1),\dots,\bar V(x,F_k)$ are linearly independent and let $\mathbb V(z)$ be the matrix with columns $\bar V(z,F_i)$, $i \in \{1,\dots,k\}$ for $z \in \interior(\A)$. Due to assumption (V\ref{ass:V0}) or (V\ref{ass:V2}), $\mathbb V(z)$ has full rank in some neighborhood $U$ of $x$. Let $r \in \{1,\dots,k\}$ and let $e_r$ be the $r$th standard unit vector of $\R^k$. We define $\lambda(z) := \mathbb V(z)^{-1} e_r$ for $z \in U$. Taking the inverse of a matrix is a continuously differentiable operation, so it is in particular locally Lipschitz continuous. Therefore, the vector $\lambda$ inherits the regularity properties of $\bar V(z,F_i)$, that is, under (V\ref{ass:V0}) $\lambda$ is continuous, and under (V\ref{ass:V2}) $\lambda$ is locally Lipschitz continuous. Therefore, these properties carry over to $h$ because for $l \in \{1,\dots,k\}$, $z \in U$
\[
h_{lr}(z) = \sum_{i=1}^k \lambda_i(z) \sum_{m=1}^k h_{lm}(z) \bar V_m(z,F_i) = \sum_{i=1}^k \lambda_i(z) \partial_l \bar S_m(z,F_i)
\]
using the assumptions on $S$. \qed

\subsubsection*{Proof of Proposition \ref{prop: integration}}

Let $x\in\interior(\A)$, $F\in\F$ and let $z\in\interior(\A)$ be some star point. Using a telescoping argument we obtain
\begin{align*}
\bar S(x,F) - \bar S(z,F) 
&= \bar S(x_1, \ldots, x_k,F) - \bar S(x_1, \ldots, x_{k-1},z_k,F) \\
&\quad+ \bar S(x_1, \ldots, x_{k-1},F) - \bar S(x_1, \ldots, x_{k-2}, z_{k-1},z_k,F) \\
&\quad+\ldots \\
&\quad+ \bar S(x_1,z_2, \ldots, z_k,F) - \bar S(z_1, \ldots,z_k,F) \\
&= \sum_{r=1}^k \int_{z_r}^{x_r} \partial_r \bar S(x_1, \ldots, x_{r-1},v,z_{r+1},\ldots, z_k,F) \diff v.
\end{align*}
Invoking the identity at \eqref{eqn: Osband FOC expectation} yields \eqref{eq:S} for the expected scores with $\bar a(F) = \bar S(z,F)$.
We denote the right hand side of \eqref{eq:S} minus $a(y)$ by $I(x,y)$, hence $\bar I(x,F) = \bar S(x,F) - \bar S(z,F)$.

For almost all $y \in \O$, the set $\{x\in \R^k \;|\; (x,y) \in C^c\} =: A_y$ has $k$-dimensional Lebesgue measure zero, where $C^c$ is the complement of the set $C$ defined in assumption (VS\ref{ass:VS}). Let $y \in \O$ be such that $A_{y}$ has measure zero. Then we obtain that for almost all $x$ the sets $\{x_i \in\R\;|\; (x,y) \in A_{y}\} =: N_i$ have one-dimensional Lebesgue-measure zero for all $i \in \{1,\dots,k\}$. Therefore, $S(x,\cdot)$ and $I(x,\cdot)$ are continuous in $y$ for almost all $x$. 

Let $(F_n)_{n \in \mathbb{N}}$ be a sequence as in assumption (F\ref{ass:F0}), that is, $(F_n)_{n \in \mathbb{N}}$ converges weakly to $\delta_{y}$ and the support of all $F_n$ is contained in some compact set $K$.  Let $\ph$ be a function on $\O$ which is locally bounded and continuous at $y$. 
By the dominated convergence theorem and the continuous mapping theorem we get that then $\int_{\O} \ph \diff F_n \to \ph(y)$. 

By this argument (recalling that $S(x,\cdot)$, $V(x,\cdot)$ are assumed to be locally bounded), if $S(x,\cdot)$ and $I(x,\cdot)$ are continuous at $y$, then $\bar S(x,F_n) - \bar I(x,F_n) \to S(x,y) - I(x,y)$. We have shown that $\bar S(x,F_n) - \bar I(x,F_n)$ does not depend on $x$, hence the same is true for the limit. Therefore, we can define $a(y) = S(x,y) - I(x,y)$ for almost all $y$. The function $a$ is $\F$-integrable, since $S$ and $I$ are $\F$-integrable.
\qed

\subsubsection*{Proof of Proposition \ref{prop:elicitable components}}

It is clear that $V$ given at \eqref{eq:V1} is a strict $\F$-identification function for $T$. Also the orientation of $V$ follows directly from its form and the orientation of its components. We have that $\partial_l \bar V_r(x,F) = 0$ for all $l,r \in\{1, \ldots, k\}$, $l\neq r$, and $x\in\interior(\A)$, $F\in\F$. 
Equation \eqref{eqn: Osband SOC2} evaluated at $x=t=T(F)$ yields
\be{eq:cross-identity}
h_{rl}(t)\partial_l \bar V_l(t,F) = h_{lr}(t) \partial_r \bar V_r(t,F).
\ee
If (V\ref{ass:V4}) holds then \eqref{eq:cross-identity} implies that $h_{rl}(t)=0$ for $r\neq l$, hence we obtain \eqref{eq:h_rl} with the surjectivity of $T$. On the other hand, if (V\ref{ass:V5}) holds, \eqref{eq:cross-identity} implies that $h_{rl}(t) = h_{lr}(t)$, whence the second part of \eqref{eq:cyclic} is shown, again using the surjectivity of $T$.
In both cases, \eqref{eqn: Osband SOC2} is equivalent to
\begin{equation}\label{eq:sim4}
\sum_{m=1}^k  \big(\partial_l h_{rm}(x) - \partial_r h_{lm}(x)\big) \bar V_m(x,F) = 0.
\end{equation}
Using assumption (V\ref{ass:V1}) there are $F_1,\dots,F_k \in \F$ such that $\bar V(x,F_1),\dots,\bar V(x,F_k)$ are linearly independent. This yields that $\partial_l h_{rm}(x) = \partial_r h_{lm}(x)$ for almost all $x \in \interior(\A)$. For the first part of the Proposition, we can conclude that $\partial_l h_{rr}(x) = \partial_r h_{lr}(x) =0$ for $r \not= l$ for almost all $x\in\interior(\A)$. Consequently, invoking that $\A$ is connected, the functions $h_{mm}$ only depend on $x_m$ and we can write $h_{mm}(x) = g_m(x_m)$ for some function $g_m\colon \A_m' \to \R$. By Lemma \ref{lem:sufficiency} (i), for $v \in \mathbb{S}^{k-1}$, $t = T(F) \in \interior(\A)$, the function $s \mapsto \bar S(t+ sv,F)$ has a global unique minimum at $s = 0$, hence 
\[
v^{\top} \nabla \bar S(t + sv,F) = \sum_{m=1}^k g_{m}(t_m + sv_m) \bar V_m (t_m + sv_m, F) v_m
\]
vanishes for $s = 0$, is negative for $s < 0$ and positive for $s > 0$, where $s$ is in some neighborhood of zero. Choosing $v$ as the $l$th standard basis vector of $\R^k$ we obtain that $g_{l} > 0$ exploiting the orientation of $V_l$ and the surjectivity of $T$.

For the second part of the proposition, to show the assertion about the definiteness, observe that due to assumption (V\ref{ass:V5}), we have for $v \in \mathbb{S}^{k-1}$, $t = T(F) \in \interior(\A)$ that $\bar V_m(t + sv,F) = c_F s v_m$ where $c_F > 0$ due to assumption (V\ref{ass:V5}) and the orientation of each component of $V$. Hence, $v^{\top} \nabla \bar S(t + sv,F) = c_F \,s v^{\top} h(t + sv) v$,
which implies the claim using again the surjectivity of $T$.
\qed

\subsubsection*{Proof of Corollary \ref{cor:elicitable components}}

The sufficiency is immediate; see the proof of Lemma \ref{lem: sum}. For the necessity, we apply Proposition \ref{prop: integration} and Proposition \ref{prop:elicitable components} to obtain that there are positive functions $g_m$ and an $\F$-integrable function $a$ such that 
\[
S(x,y) = \sum_{m=1}^k \int_{z_m}^{x_m} g_m(v) V_m(v,y)\diff v + a(y),
\]
for almost all $(x,y) \in \A \times \O$, where $z\in\interior(\A)$ is a star point of $\interior(\A)$. Let $t = T(F)$ and $x_m \not=t_m$. The strict consistency of $S$ implies that $\bar S(t,F) < \bar S(t_1,\dots,t_{m-1},x_m,t_{m+1},\dots,t_m)$. This means $\bar S_m(t_m,F) < \bar S_m(x_m,F)$ with $S_m(x_m,y):= \int_{z_m}^{x_m} g_m(v) V_m(v,y)\diff v + \frac{1}{k}a(y)$.
\qed

\subsubsection*{Proof of Theorem \ref{thm:spectral}}

(i) The second part of Theorem \ref{thm:spectral} (ii) implies the $k$-elicitability of $T$.

(ii) Let $S\colon\A\times \R\to\R$ be of the form \eqref{eq:S spectral}, $\mathcal G_k$ be convex and the functions at \eqref{eq:increasing} be increasing. Let $F\in\F$, $x= (x_1, \ldots, x_k)\in\A$ and set $t = (t_1, \ldots, t_k) = T(F)$, $w=\min(x_k,t_k)$. Then, we obtain 
\begin{align*}
&S(x,y) = \\&= \sum_{r=1}^{k-1} \big(\one\{y\le x_r\} - q_r \big) \left(G_r(x_r) + \frac{p_r}{q_r} G_k(w)(x_r - y)\right) - \one\{y\le x_r\} G_r(y)\\ 
&\quad+\big(G_k(x_k) - G_k(w)\big) \left(x_k + \sum_{m=1}^{k-1} \frac{p_m}{q_m} \big(\one\{y\le x_m\}(x_m -y) - q_mx_m \big)\right) \\ 
&\quad - \mathcal G_k(x_k)  + G_k(w)(x_k - y)+ a(y).
\end{align*}
This implies that $\bar S(x,F) - \bar S(t,F) = R_1+ R_2$ with 
\begin{align*}
R_1 &= \sum_{r=1}^{k-1} \big(F(x_r) - q_r\big) \left(G_r(x_r) +  \frac{p_r}{q_r} G_k(w)x_r \right)  \\&\qquad-\int_{t_r}^{x_r} \left(G_r(y) +  \frac{p_r}{q_r} G_k(w)y \right)\diff F(y), \\
R_2 &=
\big(G_k(x_k) - G_k(w)\big) \left(x_k + \sum_{m=1}^{k-1} \frac{p_m}{q_m} \left(\int_{-\infty}^{x_m}(x_m -y)\diff F(y) - q_mx_m \right)\right) \\
& \quad - \mathcal G_k(x_k) + \mathcal G_k(t_k) + G_k(w)(x_k - t_k).
\end{align*}
We denote the $r$th summand of $R_1$ by $\xi_r$ and suppose that $t_r<x_r$. Due to the assumptions, the term $G_r(y) +  \frac{p_r}{q_r} G_k(w)y $ is increasing in $y\in [t_r, x_r]$ which implies that 
\begin{align*}
\xi_r&\ge \big(F(x_r) - q_r\big) \left(G_r(x_r) +  \frac{p_r}{q_r} G_k(w)x_r \right)  \\
&\quad -\big(F(x_r) - F(t_r)\big) \left(G_r(x_r) +  \frac{p_r}{q_r} G_k(w)x_r \right)=0.
\end{align*}
Analogously, one can show that $\xi_r\ge0$ if $x_r<t_r$. If $F$ has a unique $q_r$-quantile and the term $G_r(y) +  \frac{p_r}{q_r} G_k(w)y $ is strictly increasing in $y$, then we even get $\xi_r>0$ if $x_r\neq t_r$.

Now consider the term $R_2$. Splitting the integrals from $\infty$ to $x_m$ into integrals from $-\infty$ to $t_m$ and from $t_m$ to $x_m$ and partially integrating the latter, we obtain
\begin{align*}
R_2 &= \big(G_k(x_k) - G_k(w)\big)  \left(x_k + \sum_{m=1}^{k-1} p_m \left(t_m-x_m - \frac{1}{q_m}\int_{-\infty}^{t_m} y\diff F(y) + \frac{1}{q_m}\int_{t_m}^{x_m}F(y)\diff y \right)\right) \\
& \quad - \mathcal G_k(x_k) + \mathcal G_k(t_k) + G_k(w)(x_k - t_k) \\
&= \big(G_k(x_k) - G_k(w)\big)\left(x_k - t_k +\sum_{m=1}^{k-1} p_m \left(t_m-x_m  + \frac{1}{q_m}\int_{t_m}^{x_m}F(y)\diff y \right) \right)  \\
& \quad - \mathcal G_k(x_k) + \mathcal G_k(t_k) + G_k(w)(x_k - t_k) \\
&\ge \big(G_k(x_k) - G_k(w)\big) (x_k - t_k) - \mathcal G_k(x_k) + \mathcal G_k(t_k) + G_k(w)(x_k - t_k)\\
&= \mathcal G_k(t_k) -\mathcal G_k(x_k) - G_k(x_k)(t_k - x_k) \ge0.
\end{align*}
The first inequality is due to the fact that (i) $G_k$ is increasing and (ii) for $x_m\neq t_m$ we have $\frac{1}{q_m}\int_{t_m}^{x_m} F(y)\diff y\ge x_m - t_m$ with strict inequality if $F$ has a unique $q_m$-quantile. The last inequality is due to the fact that $\mathcal G_k$ is convex. The inequality is strict if $x_k\neq t_k$ and if $\mathcal G_k$ is strictly convex.

(iii) If $f$ denotes the density of $F$, it holds that
\be{eq:U_alpha}
\ES_{\alpha}(F) = \frac{1}{\alpha}\int_{-\infty}^{F^{-1}(\alpha)} yf(y)\diff y, \qquad \alpha\in(0,1].
\ee

We first show the assertions concerning $V$ given at \eqref{eq:V spectral}. Let $F\in \F$ with density $f = F'$ and let $t = T(F)$. Then we have for $m\in\{1, \ldots, k-1\}$, $x\in\A$, that $\bar V_m(x,F) = F(x_m) - q_m$ which is zero if and only if $x_m = t_m$. On the other hand, using the identity at \eqref{eq:U_alpha} 
\[
\bar V_k(t_1, \ldots, t_{k-1},x_k,F) = x_k - \sum_{m=1}^{k-1} \frac{p_m}{q_m} \int_{-\infty}^{t_m} yf(y)\diff y = x_k - t_k.
\]
Hence, it follows that $V$ is a strict $\F$-identification function for $T$. Moreover, $V$ satisfies assumption  (V\ref{ass:V2}), and we have for $m\in\{1, \ldots, k-1\}$, $l\in\{1, \ldots, k\}$ and $x\in\interior(\A)$ that $\partial_l \bar V_m(x,F)=0$ if $l \neq m$ and $\partial_m \bar V_m(x,F)= f(x_m)$, $\partial_m \bar V_k(x,F)= -(p_m/q_m)x_m f(x_m)$ and $\partial_k \bar V_k(x,F) = 1$. 

From now on, we assume that $t = T(F)\in \interior(\A)$.
Let $S$ be a strictly $\F$-consistent scoring function for $T$ satisfying (S\ref{ass:S3}). Then we can apply Theorem \ref{thm: Osband's principle}  and Corollary \ref{cor: OP SOC}  to get that there are locally Lipschitz continuous functions $h_{lm}\colon \interior(\A)\to\R$ such that \eqref{eqn: Osband FOC expectation} and \eqref{eqn: Osband SOC2} hold. If we evaluate \eqref{eqn: Osband SOC2} for $l=k$, $m\in\{1,\ldots, k-1\}$ at the point $x=t$ we get
\[
h_{km}(t)\partial_m \bar V_m(t,F) + h_{kk}(t) \partial_m\bar V_k(t,F) = h_{mk}(t)\partial_k\bar V_k(t,F),
\]
which takes the form $h_{km}(t) f(t_m) - h_{kk}(t) \frac{p_m}{q_m}t_m f(t_m) = h_{mk}(t)$.
Invoking assumption (V\ref{ass:V4}) for $(V_1,\ldots, V_{k-1})$, we get that necessarily $h_{mk}(t)=0$ and $h_{km}(t) = (p_m/q_m)t_m h_{kk}(t)$. So with the surjectivity of $T$ we get for $x\in\interior(\A)$ that 
\begin{align}\label{eq:1}
h_{mk}(x)=0, && h_{km}(x) = \frac{p_m}{q_m} x_m h_{kk}(x) &&\text{for all } m\in\{1, \ldots, k-1\}.
\end{align}
Now, we can evaluate \eqref{eqn: Osband SOC2} for $m,l\in\{1,\ldots, k-1\}$, $m\neq l$, at $x=t$ and use the first part of \eqref{eq:1} to get that $h_{ml}(t) f(t_l) = h_{lm}(t) f(t_m)$.
Using again the same argument, we get for $x \in \interior(\A)$ that
\be{eq:2}
h_{ml}(x)=0 \qquad \text{for all } m,l\in\{1, \ldots, k-1\}, \ l\neq m.
\ee
At this stage, we can evaluate \eqref{eqn: Osband SOC2} for $l\in\{1, \ldots, k-1\}$, $m \in \{1,\dots,k\}$, $m\neq l$, for some $x\in\interior(\A)$. Using \eqref{eq:1} and \eqref{eq:2} we obtain
\[
\sum_{i=1}^k \big(\partial_l h_{mi}(x) - \partial_m h_{li}(x)\big)\bar V_i(x_i,F) = 0.
\] 
Invoking assumption (V\ref{ass:V1}) and using \eqref{eq:1} and \eqref{eq:2}, we can conclude that for almost all $x \in \A$,
\be{eq:3}
\partial_l h_{mm}(x) =0 \qquad \text{for all } l\in\{1, \ldots, k-1\},\  m \in \{1,\dots,k\}, \ l\neq r.
\ee
and 
\begin{equation}\label{eq:4}
\partial_k h_{ll}(x) = \frac{p_l}{q_l} h_{kk}(x) \qquad \text{for all } l\in\{1, \ldots, k-1\}.
\end{equation}
Equation \eqref{eq:3} for $m = k$ shows that there is a locally Lipschitz continuous function $g_k \colon \A_k'\to\R$ such that for all $(x_1, \ldots, x_k)\in \interior(\A)$, we have $h_{kk}(x_1, \ldots, x_k) = g_k(x_k)$.
Equation \eqref{eq:4} together with \eqref{eq:3} gives that  for $l \in \{1,\dots,k-1\}$, and $(x_1, \ldots, x_k)\in \interior(\A)$, we obtain $h_{ll}(x_1, \ldots, x_k)=(p_l/q_l) G_k(x_k) + g_l(x_l)$, where $g_l\colon\A_l'\to\R$ is locally Lipschitz continuous and $G_k\colon \A_k'\to\R$ is such that $G_k'=g_k$.

Knowing the form of the matrix-valued function $h$, we can apply the second part of Proposition \ref{prop: integration}. Let $z\in\interior(\A)$ be some star point. Then there is some $\F$-integrable function $b\colon \R\to\R$ such that 
\begin{align}\label{eq:S integrated}
S(x,y) &=
 \sum_{r=1}^{k-1} \int_{z_r}^{x_r} \left(\frac{p_r}{q_r} G_k(z_k) +g_r(v)\right) \big(\one\{y\le v\} - q_r \big)\diff v \\ \nonumber
&+\big(G_k(x_k) - G_k(z_k)\big) \sum_{m=1}^{k-1} \frac{p_m}{q_m} \big(x_m(\one\{y\le x_m\} -q_m) - y\one\{y\le x_m\} \big) \\ 
& + G_k(x_k)x_k - \mathcal G_k(x_k) + b(y),\nonumber
\end{align}
for almost all $(x,y)$ where $\mathcal G_k\colon \A_k'\to\R$ is such that $\mathcal G_k'=G_k$. One can check by a straightforward computation that the representation of $S$ at \eqref{eq:S integrated} is equivalent to the one at \eqref{eq:S spectral} upon choosing a suitable $\F$-integrable function $a\colon \R\to\R$.  

It remains to show that $\mathcal G_k$ is strictly convex and that the functions given at \eqref{eq:increasing} are strictly increasing. To this end, we use  Lemma \ref{lem:sufficiency} part (i). Let $D = \{s \in \R \colon t+ sv \in \interior(\A)\}$, and let $v = (v_1,\dots,v_k) \in \mathbb{S}^{k-1}$ and without loss of generality assume $v_k \ge 0$. We define $\psi\colon D \to \R$ by $\psi(s):= \bar S(t + sv,F)$, that is,
\begin{equation*}\label{eq:psis}
\begin{split}
\psi(s)&= \sum_{r=1}^{k-1}\int_{z_r}^{\bar s_r}\Big(\frac{p_r}{q_r}G_k(z_k) + g_r(v)\Big)(F(v) - q_r) \dint v \\
&\quad + (G_k(\bar s_k)-G_k(z_k)) \sum_{m=1}^{k-1}\frac{p_m}{q_m} \Big(\bar s_m (F(\bar s_m) - q_m) - \int_{-\infty}^{\bar s_m} y f(y)\dint y\Big)\\
&\quad + \bar s_k G_k(\bar s_k) - \mathcal{G}_k(\bar s_k) + \bar b(F) , 
\end{split}
\end{equation*}
where we use the notation $\bar s = t + sv$. The function $\psi$ has a minimum at $s =0$. Hence, there is $\varepsilon>0$ such that $\psi'(s)<0$ for $s\in(-\varepsilon,0)$ and $\psi'(s)>0$ for $s\in(0,\varepsilon)$. If $v_k = 0$, then
\begin{equation*}\label{eq:dpsi-vkzero}
\psi'(s) = \sum_{r=1}^{k-1} (F(\bar s_r) - q_r)v_r\Big(g_r(\bar s_r) + \frac{p_r}{q_r}G_k(\bar s_k)\Big).
\end{equation*}
Choosing $v$ as the $m$th standard basis vector of $\R^k$ for $m \in \{1,\dots,k-1\}$, we obtain that $g_r(\bar s_r) + \frac{p_r}{q_r}G_k(\bar s_k)>0$. Exploiting the surjectivity of $T$ we can deduce that the functions at \eqref{eq:increasing} are strictly increasing. On the other hand, if $v$ is the $k$th standard basis vector, we obtain that $\psi'(s) = g_k(\bar s_k) s$. Again using the surjectivity of $T$ we get that $g_k>0$ which shows the strict convexity of $\mathcal G_k$.
\qed

\subsubsection*{Proof of Corollary \ref{cor:spectral}}

For the first part of the claim, note that if $\mu(\{1\})=1$, then $\nu_\mu$ coincides with the expectation and is thus 1-elicitable. If $\mu(\{1\})=0$, the assertion of the corollary is a direct consequence of Theorem \ref{thm:spectral} (i). If $\lambda:=\mu(\{1\})\in(0,1)$, then we can write
\(
\mu = \sum_{m=1}^{k-2} p_m \delta_{q_m} + \lambda \delta_1
\),
where $p_m\in(0,1)$, $\sum_{m=1}^{k-2} p_m=1-\lambda$, $q_m\in(0,1)$ and the $q_m$'s are pairwise distinct. Define the probability measure
\(
\tilde \mu:= \sum_{m=1}^{k-2} \frac{p_m}{1-\lambda} \delta_{q_m}
\).
Using Theorem \ref{thm:spectral} (i), the functional $(T_1',\ldots, T_{k-1}')\colon \F\to\R^{k-1}$ is $(k-1)$-elicitable where $T_m'(F):=F^{-1}(q_m)$, $m\in\{1,\ldots, k-2\}$, and $T_{k-1}'(F) = \nu_{\tilde\mu}(F)$.
Using Lemma \ref{lem: sum} we can deduce that the functional $(T_1',\ldots, T_{k-1}',\nu_{\delta_1}) \colon$ $\F\to \R^k$ is $k$-elicitable. Note that 
\(
\nu_{\mu} = (1-\lambda)\nu_{\tilde\mu} + \lambda \nu_{\delta_1}.
\)
Hence, we can apply Proposition \ref{prop: revelation principle} to deduce that the functional $T = (T_1,\ldots, T_k) : \F\to\R^k$ is $k$-elicitable where $T_m = T_m'$, $m\in\{1,\ldots, k-2\}$, $T_{k-1} = \nu_{\delta_1}$ and $T_k = \nu_\mu$.
\qed

\subsubsection*{Proof of Corollary \ref{cor:VaR-ES}}
The sufficiency follows directly from Theorem \ref{thm:spectral}.
We will show that $G_2$ is necessarily bounded below. Suppose the contrary. For the action domain $\A_0$, we have $\A'_{1,x_2} = [x_2,\infty)$, therefore, for $x_2 \le x_1 < x_1'$ \eqref{eq:increasing} yields
\[
-\infty < G_1(x_1) - G_1(x_1') \le \frac{1}{\alpha}G_2(x_2)(x_1' - x_1).
\]
Letting $x_2 \to -\infty$ one obtains a contradiction. Let $C_2=\lim_{x_2 \to -\infty}G_2(x_2) > -\infty$. Then, by \eqref{eq:increasing}, we obtain that $G_1(x_1) + (C_2/\alpha)x_1$ is increasing in $x_1 \in \R$. We can write $S$ at \eqref{eq:S VaR,ES} as
\begin{align*}
S(x_1,x_2,y) &=\big(\one\{y\le x_1\} - \alpha \big)\left(G_1(x_1) + \frac{C_2}{\alpha}x_1\right) - \one\{y\le x_1\}\left(G_1(y)+\frac{C_2}{\alpha}y\right)\\
&\quad + (G_2(x_2) - C_2)\Big(\frac{1}{\alpha}\one\{y\le x_1\}(x_1 - y) - (x_1 - x_2)\Big) \\
&\quad - (\mathcal{G}_2(x_2)-C_2x_2) + a(y).
\end{align*}
The last expression is again of the form at \eqref{eq:S VaR,ES} with an increasing function $\tilde{G}_1(x_1) = G_1(x_1) + (C_2/\alpha)x_1$ and with $\tilde{G}_2(x_2) = G_2(x_2) - C_2 \ge 0$. 
\qed

\section*{Acknowledgements}
{We would like to thank Carlo Acerbi, Rafael Fron\-gillo and Tilmann Gneiting for fruitful discussions which inspired some of the results of this paper. We are grateful for the valuable comments of two anonymous referees which significantly improved the paper. This project is supported by the Swiss National Science Foundation (SNF) via grant 152609.}

\bibliographystyle{plainnat}


\end{document}